\documentclass[12pt]{amsart}
\usepackage{txfonts}      
\usepackage{amssymb}
\usepackage{eucal}
\usepackage{amsmath}
\usepackage{amscd}
\usepackage{xcolor}
\usepackage{multicol}
\usepackage[all]{xy}           
\usepackage{graphicx}
\usepackage{color}
\usepackage{colordvi}
\usepackage{xspace}
\usepackage{tikz}
\usepackage{makecell}
\usepackage{appendix}
\usepackage{amsthm}
\usepackage[misc]{ifsym}
\usepackage{mathrsfs}

\usepackage{ifpdf}
\ifpdf
\usepackage[colorlinks,final,backref=page,hyperindex]{hyperref}
\else
\usepackage[colorlinks,final,backref=page,hyperindex,hypertex]{hyperref}
\fi

\usepackage[active]{srcltx} 

\begin{document}
\newtheorem{theorem}{Theorem}[section]
\newtheorem{lemma}[theorem]{Lemma}
\newtheorem{corollary}[theorem]{Corollary}
\newtheorem{proposition}[theorem]{Proposition}
\theoremstyle{definition}
\newtheorem{definition}[theorem]{Definition}
\newtheorem{example}[theorem]{Example}
\newtheorem{remark}[theorem]{Remark}
\newtheorem{pdef}[theorem]{Proposition-Definition}
\newtheorem{condition}[theorem]{Condition}
\renewcommand{\labelenumi}{{\rm(\alph{enumi})}}
\renewcommand{\theenumi}{\alph{enumi}}
\baselineskip=14pt

\newcommand {\emptycomment}[1]{} 

\newcommand{\nc}{\newcommand}
\newcommand{\delete}[1]{}

\nc{\todo}[1]{\tred{To do:} #1}

\nc{\tred}[1]{\textcolor{red}{#1}}
\nc{\tblue}[1]{\textcolor{blue}{#1}}
\nc{\tgreen}[1]{\textcolor{green}{#1}}
\nc{\tpurple}[1]{\textcolor{purple}{#1}}
\nc{\tgray}[1]{\textcolor{gray}{#1}}
\nc{\torg}[1]{\textcolor{orange}{#1}}
\nc{\tmag}[1]{\textcolor{magenta}}
\nc{\btred}[1]{\textcolor{red}{\bf #1}}
\nc{\btblue}[1]{\textcolor{blue}{\bf #1}}
\nc{\btgreen}[1]{\textcolor{green}{\bf #1}}
\nc{\btpurple}[1]{\textcolor{purple}{\bf #1}}

	\nc{\mlabel}[1]{\label{#1}}  
	\nc{\mcite}[1]{\cite{#1}}  
	\nc{\mref}[1]{\ref{#1}}  
	\nc{\meqref}[1]{\eqref{#1}}  
	\nc{\mbibitem}[1]{\bibitem{#1}} 

\delete{
	\nc{\mlabel}[1]{\label{#1}  
		{ {\small\tgreen{\tt{{\ }(#1)}}}}}
	\nc{\mcite}[1]{\cite{#1}{\small{\tt{{\ }(#1)}}}}  
	\nc{\mref}[1]{\ref{#1}{\small{\tred{\tt{{\ }(#1)}}}}}  
	\nc{\meqref}[1]{\eqref{#1}{{\tt{{\ }(#1)}}}}  
	\nc{\mbibitem}[1]{\bibitem[\bf #1]{#1}} 
}

\nc{\cm}[1]{\textcolor{red}{Chengming:#1}}
\nc{\yy}[1]{\textcolor{blue}{Yanyong: #1}}
\nc{\zy}[1]{\textcolor{yellow}{Zhongyin: #1}}
\nc{\li}[1]{\textcolor{purple}{#1}}
\nc{\lir}[1]{\textcolor{purple}{Li:#1}}


\nc{\tforall}{\ \ \text{for all }}
\nc{\hatot}{\,\widehat{\otimes} \,}
\nc{\complete}{completed\xspace}
\nc{\wdhat}[1]{\widehat{#1}}

\nc{\ts}{\mathfrak{p}}
\nc{\mts}{c_{(i)}\ot d_{(j)}}

\nc{\NA}{{\bf NA}}
\nc{\LA}{{\bf Lie}}
\nc{\CLA}{{\bf CLA}}

\nc{\cybe}{CYBE\xspace}
\nc{\nybe}{NYBE\xspace}
\nc{\ccybe}{CCYBE\xspace}

\nc{\ndend}{pre-Novikov\xspace}
\nc{\calb}{\mathcal{B}}
\nc{\rk}{\mathrm{r}}
\newcommand{\g}{\mathfrak g}
\newcommand{\h}{\mathfrak h}
\newcommand{\pf}{\noindent{$Proof$.}\ }
\newcommand{\frkg}{\mathfrak g}
\newcommand{\frkh}{\mathfrak h}
\newcommand{\Id}{\rm{Id}}
\newcommand{\gl}{\mathfrak {gl}}
\newcommand{\ad}{\mathrm{ad}}
\newcommand{\add}{\frka\frkd}
\newcommand{\frka}{\mathfrak a}
\newcommand{\frkb}{\mathfrak b}
\newcommand{\frkc}{\mathfrak c}
\newcommand{\frkd}{\mathfrak d}
\newcommand {\comment}[1]{{\marginpar{*}\scriptsize\textbf{Comments:} #1}}


\nc{\disp}[1]{\displaystyle{#1}}
\nc{\bin}[2]{ (_{\stackrel{\scs{#1}}{\scs{#2}}})}  
\nc{\binc}[2]{ \left (\!\! \begin{array}{c} \scs{#1}\\
    \scs{#2} \end{array}\!\! \right )}  
\nc{\bincc}[2]{  \left ( {\scs{#1} \atop
    \vspace{-.5cm}\scs{#2}} \right )}  
\nc{\ot}{\otimes}
\nc{\sot}{{\scriptstyle{\ot}}}
\nc{\otm}{\overline{\ot}}
\nc{\ola}[1]{\stackrel{#1}{\la}}

\nc{\scs}[1]{\scriptstyle{#1}} \nc{\mrm}[1]{{\rm #1}}

\nc{\dirlim}{\displaystyle{\lim_{\longrightarrow}}\,}
\nc{\invlim}{\displaystyle{\lim_{\longleftarrow}}\,}

\nc{\bfk}{{\bf k}} \nc{\bfone}{{\bf 1}}
\nc{\rpr}{\circ}
\nc{\dpr}{{\tiny\diamond}}
\nc{\rprpm}{{\rpr}}

\nc{\mmbox}[1]{\mbox{\ #1\ }} \nc{\ann}{\mrm{ann}}
\nc{\Aut}{\mrm{Aut}} \nc{\can}{\mrm{can}}
\nc{\twoalg}{{two-sided algebra}\xspace}
\nc{\colim}{\mrm{colim}}
\nc{\Cont}{\mrm{Cont}} \nc{\rchar}{\mrm{char}}
\nc{\cok}{\mrm{coker}} \nc{\dtf}{{R-{\rm tf}}} \nc{\dtor}{{R-{\rm
tor}}}
\renewcommand{\det}{\mrm{det}}
\nc{\depth}{{\mrm d}}
\nc{\End}{\mrm{End}} \nc{\Ext}{\mrm{Ext}}
\nc{\Fil}{\mrm{Fil}} \nc{\Frob}{\mrm{Frob}} \nc{\Gal}{\mrm{Gal}}
\nc{\GL}{\mrm{GL}} \nc{\Hom}{\mrm{Hom}} \nc{\hsr}{\mrm{H}}
\nc{\hpol}{\mrm{HP}}  \nc{\id}{\mrm{id}} \nc{\im}{\mrm{im}}

\nc{\incl}{\mrm{incl}} \nc{\length}{\mrm{length}}
\nc{\LR}{\mrm{LR}} \nc{\mchar}{\rm char} \nc{\NC}{\mrm{NC}}
\nc{\mpart}{\mrm{part}} \nc{\pl}{\mrm{PL}}
\nc{\ql}{{\QQ_\ell}} \nc{\qp}{{\QQ_p}}
\nc{\rank}{\mrm{rank}} \nc{\rba}{\rm{RBA }} \nc{\rbas}{\rm{RBAs }}
\nc{\rbpl}{\mrm{RBPL}}
\nc{\rbw}{\rm{RBW }} \nc{\rbws}{\rm{RBWs }} \nc{\rcot}{\mrm{cot}}
\nc{\rest}{\rm{controlled}\xspace}
\nc{\rdef}{\mrm{def}} \nc{\rdiv}{{\rm div}} \nc{\rtf}{{\rm tf}}
\nc{\rtor}{{\rm tor}} \nc{\res}{\mrm{res}} \nc{\SL}{\mrm{SL}}
\nc{\Spec}{\mrm{Spec}} \nc{\tor}{\mrm{tor}} \nc{\Tr}{\mrm{Tr}}
\nc{\mtr}{\mrm{sk}}

\nc{\ab}{\mathbf{Ab}} \nc{\Alg}{\mathbf{Alg}}

\nc{\BA}{{\mathbb A}} \nc{\CC}{{\mathbb C}} \nc{\DD}{{\mathbb D}}
\nc{\EE}{{\mathbb E}} \nc{\FF}{{\mathbb F}} \nc{\GG}{{\mathbb G}}
\nc{\HH}{{\mathbb H}} \nc{\LL}{{\mathbb L}} \nc{\NN}{{\mathbb N}}
\nc{\QQ}{{\mathbb Q}} \nc{\RR}{{\mathbb R}} \nc{\BS}{{\mathbb{S}}} \nc{\TT}{{\mathbb T}}
\nc{\VV}{{\mathbb V}} \nc{\ZZ}{{\mathbb Z}}


\nc{\calao}{{\mathcal A}} \nc{\cala}{{\mathcal A}}
\nc{\calc}{{\mathcal C}} \nc{\cald}{{\mathcal D}}
\nc{\cale}{{\mathcal E}} \nc{\calf}{{\mathcal F}}
\nc{\calfr}{{{\mathcal F}^{\,r}}} \nc{\calfo}{{\mathcal F}^0}
\nc{\calfro}{{\mathcal F}^{\,r,0}} \nc{\oF}{\overline{F}}
\nc{\calg}{{\mathcal G}} \nc{\calh}{{\mathcal H}}
\nc{\cali}{{\mathcal I}} \nc{\calj}{{\mathcal J}}
\nc{\call}{{\mathcal L}} \nc{\calm}{{\mathcal M}}
\nc{\caln}{{\mathcal N}} \nc{\calo}{{\mathcal O}}
\nc{\calp}{{\mathcal P}} \nc{\calq}{{\mathcal Q}} \nc{\calr}{{\mathcal R}}
\nc{\calt}{{\mathcal T}} \nc{\caltr}{{\mathcal T}^{\,r}}
\nc{\calu}{{\mathcal U}} \nc{\calv}{{\mathcal V}}
\nc{\calw}{{\mathcal W}} \nc{\calx}{{\mathcal X}}
\nc{\CA}{\mathcal{A}}

\nc{\fraka}{{\mathfrak a}} \nc{\frakB}{{\mathfrak B}}
\nc{\frakb}{{\mathfrak b}} \nc{\frakd}{{\mathfrak d}}
\nc{\oD}{\overline{D}}
\nc{\frakF}{{\mathfrak F}} \nc{\frakg}{{\mathfrak g}}
\nc{\frakm}{{\mathfrak m}} \nc{\frakM}{{\mathfrak M}}
\nc{\frakMo}{{\mathfrak M}^0} \nc{\frakp}{{\mathfrak p}}
\nc{\frakS}{{\mathfrak S}} \nc{\frakSo}{{\mathfrak S}^0}
\nc{\fraks}{{\mathfrak s}} \nc{\os}{\overline{\fraks}}
\nc{\frakT}{{\mathfrak T}}
\nc{\oT}{\overline{T}}
\nc{\frakX}{{\mathfrak X}} \nc{\frakXo}{{\mathfrak X}^0}
\nc{\frakx}{{\mathbf x}}
\nc{\frakTx}{\frakT}      
\nc{\frakTa}{\frakT^a}        
\nc{\frakTxo}{\frakTx^0}   
\nc{\caltao}{\calt^{a,0}}   
\nc{\ox}{\overline{\frakx}} \nc{\fraky}{{\mathfrak y}}
\nc{\frakz}{{\mathfrak z}} \nc{\oX}{\overline{X}}

\font\cyr=wncyr10


\title{Quasi-Frobenius Novikov algebras and Pre-Novikov bialgebras}

\author{Yue Li}
\address{School of Mathematics, Hangzhou Normal University,
Hangzhou, 311121, China}
\email{liyuee@stu.hznu.edu.cn}

\author{Yanyong Hong}
\address{School of Mathematics, Hangzhou Normal University,
Hangzhou, 311121, China}
\email{yyhong@hznu.edu.cn}

\subjclass[2010]{17A30, 17A60, 17B38, 17D25}
\keywords{Novikov algebra, pre-Novikov algebra, pre-Novikov bialgebra, Yang-Baxter equation, $\mathcal{O}$-operator}
\footnote {
Corresponding author: Y. Hong (yyhong@hznu.edu.cn).
}
\begin{abstract}
Pre-Novikov algebras and quasi-Frobenius Novikov algebras naturally appear in the theory of Novikov bialgebras. In this paper, we show that there is a natural pre-Novikov algebra structure associated to a quasi-Frobenius Novikov algebra. Then we introduce the definition of double constructions of quasi-Frobenius Novikov algebras associated to two pre-Novikov algebras and show that it is characterized by a pre-Novikov bialgebra. We also introduce the notion of pre-Novikov Yang-Baxter equation, whose symmetric solutions can produce pre-Novikov bialgebras. Moreover, the operator forms of pre-Novikov Yang-Baxter equation are also investigated.
\end{abstract}

\maketitle
\section{Introduction}
Novikov algebras appeared in the study of Hamiltonian
operators in the formal variational calculus~\mcite{GD1, GD2} and
Poisson brackets of hydrodynamic type~\mcite{BN}. It was also shown in \mcite{X1} that Novikov algebras correspond to a class of Lie conformal algebras, which describe the singular part of operator product expansion of chiral fields in conformal field
theory \mcite{K1}. Note that Novikov algebras are also an important subclass of pre-Lie algebras (also called left-symmetric algebras), which are closely related to many fields in mathematics and physics such as  affine manifolds and affine
structures on Lie groups \mcite{Ko},  convex homogeneous cones \mcite{V}, deformation of associative algebras \mcite{Ger} and vertex algebras \mcite{BK, BLP}.

The definitions of pre-Novikov algebras and quasi-Frobenius Novikov algebras naturally appeared in the study of Novikov bialgebras \cite{HBG}. It was shown in \cite{HBG} that there is a Novikov algebra associated to a pre-Novikov algebra and pre-Novikov algebras can produce skewsymmetric solutions of Novikov Yang-Baxter equation and hence Novikov bialgebras. Moreover, by \cite{XH}, pre-Novikov algebras correspond to a class of left-symmetric conformal algebras \cite{HL} and there are close relationships between pre-Novikov algebras and Zinbiel algebras with a derivation (see \cite{HBG, KMS}). Quasi-Frobenius Novikov algebras are a class of Novikov algebras with a special bilinear form, which are closely related with the skewsymmetric solutions of Novikov Yang-Baxter equation (see \cite{HBG}). Moreover, it was shown in \cite{HBG} that quasi-Frobenius Novikov algebras also correspond to a class of quasi-Frobenius infinite-dimensional Lie algebras.

As we know, in the case of associative algebras, there is a natural construction of Frobenius algebras called a double construction of Frobenius algebras (see \cite{Bai}). It was shown in \cite{Bai} that a double construction of Frobenius algebras is characterized by an antisymmetric infinitesimal bialgebra. Motivated by this, it is natural to consider double constructions of quasi-Frobenius Novikov algebras and the relationships with the bialgebra theory of pre-Novikov algebras. Note that pre-Novikov algebras are a class of L-dendriform algebras \cite{BLN} and the theory of L-dendriform bialgebras was developed in \cite{BLN, NB}. Therefore, it is also natural to investigate the bialgebra theory of pre-Novikov algebras.

In this paper, we introduce the definitions of pre-Novikov bialgebras and double constructions of quasi-Frobenius Novikov algebras associated to two pre-Novikov algebras and show that these two definitions are equivalent. We also investigate a special class of pre-Novikov bialgebras, which resembles the coboundary Lie bialgebras \cite{CP}. The notion of pre-Novikov Yang-Baxter equation is also introduced, whose symmetric solutions give pre-Novikov bialgebras. We introduce the $\mathcal{O}$-operator on a pre-Novikov algebra as an operator form of pre-Novikov Yang-Baxter equation. It gives a symmetric solution of pre-Novikov Yang-Baxter equation in a semi-direct product pre-Novikov algebra. It should be pointed out that although pre-Novikov algebras are L-dendriform algebras, their bialgebra theories are totally different, for example, skewsymmetric solutions of LD-equation give L-dendriform bialgebras, whereas symmetric solutions of pre-Novikov Yang-Baxter equation give pre-Novikov bialgebras.

This paper is organized as follows. In Section 2, we recall some basic facts about Novikov algebras, pre-Novikov algebras and quasi-Frobenius Novikov algebras. Moreover, we show that there is a pre-Novikov algebra associated to a quasi-Frobenius Novikov algebra. In Section 3, the definition of double constructions of quasi-Frobenius Novikov algebras associated to two pre-Novikov algebras is introduced and we show that it is equivalent to a special matched pair of Novikov algebras. Moreover, we introduce the definition of pre-Novikov bialgebras and prove that it can also characterize the double construction of quasi-Frobenius Novikov algebras associated to two pre-Novikov algebras. In Section 4, we introduce the definition of pre-Novikov Yang-Baxter equation and show that symmetric solutions of pre-Novikov Yang-Baxter equation can produce pre-Novikov bialgebras. Moreover, we investigate the operator forms of pre-Novikov Yang-Baxter equation. The definition of $\mathcal{O}$-operators on pre-Novikov algebras is introduced and we show that it produces a symmetric solution of pre-Novikov Yang-Baxter equation in a semi-direct product pre-Novikov algebra.

\smallskip
\noindent
{\bf Notations.}
Throughout this paper, we fix a base field ${\bf k}$ of characteristic $0$. All vector spaces and algebras are over $\bfk$. Unless otherwise stated, they  are assumed to be finite-dimensional even though many results still hold in the infinite-dimensional cases. The identity map is denoted by $\id$. Let $A$ be a vector space with a binary operation $\circ$.
Define linear maps
$L_{\circ}, R_{\circ}:A\rightarrow {\rm End}_{\bf k}(A)$ by
\begin{eqnarray*}
L_{\circ}(a)b:=a\circ b,\;\; R_{\circ}(a)b:=b\circ a, \;\;\;a, b\in A.
\end{eqnarray*}
Let
$$\tau:A\otimes A\rightarrow A\otimes A,\quad a\otimes b\mapsto b\otimes a,\;\;\;\; a,b\in A,$$
be the flip operator.

\delete{
Throughout this paper, all algebras are finite dimensional and over a field $\mathbf{k}$ of characteristic zero. We give some notation as follows. Let $V,W$ be vector spaces.\\
\hspace*{4mm}(a) Let $(V,\circ)$ be a vector space with a binary operation $\circ$. We define the left action and right action by $L_\circ :V\rightarrow \text{End}_{\bf k}(V)$ and $R_\circ :V\rightarrow \text{End}_{\bf k}(V)$ respectively,
$$L_\circ (x)y=x\circ y=R_\circ (y)x,\quad x,y\in V.$$
\hspace*{4mm}(b) The exchanging operator $\tau:V\otimes V\rightarrow V\otimes V $ is defined by
\begin{align*}
\tau(x\otimes y)=y\otimes x,\quad  x,y\in V.
\end{align*}}

\section{Preliminaries}
In this section, we recall some known facts about Novikov algebras and pre-Novikov algebras and show that there is a natural pre-Novikov algebra structure associated to a quasi-Frobenius Novikov algebra.

Recall that a {\bf Novikov algebra} is a vector space $A$ with a binary operation $\circ :A\otimes A \rightarrow A$ satisfying
\begin{align}
(a\circ b)\circ c-a\circ (b\circ c)&=(b\circ a)\circ c-b\circ (a\circ c),\\
(a\circ b)\circ c&=(a\circ c)\circ b,\quad  a,b,c\in A.
\end{align}
Denote it by $(A,\circ)$.

\begin{definition}\cite{O}
A {\bf representation} of a Novikov algebra $(A,\circ)$ is a triple $(V,l,r)$, where $V$ is a vector space and $l,r: A\rightarrow \text{End}_{\bf k}(V)$ are linear maps satisfying
\begin{align}
l(a\circ b-b\circ a)v&=l(a)l(b)v-l(b)l(a)v,\\
l(a)r(b)v-r(b)l(a)v&=r(a\circ b)v-r(b)r(a)v,\\
l(a\circ b)v&=r(b)l(a)v,\\
r(a)r(b)v&=r(b)r(a)v, \quad a,b\in A,\;\;v\in V.
\end{align}

\end{definition}

\begin{remark}
Obviously, $(A,L_{\circ},R_{\circ})$ is a representation of $(A,\circ)$, which is called the \textbf{adjoint representation} of $(A,\circ)$.
\end{remark}

Let $(A,\circ)$ be a vector space with a binary operation $\circ$ and $V$ be a vector space. For a linear map $\rho:A\rightarrow \text{End}_{\bf k}(V)$,
define a linear map $\rho^*: A\rightarrow \text{End}_{\bf k}(V^*)$ by
\begin{align}
\langle \rho^*(a)f,v\rangle&=-\langle f,\rho(a)v\rangle,\quad  a\in A, v\in V, f\in V^*,\label{e14}
\end{align}
where $\langle \cdot,\cdot\rangle$ is the usual pairing between  $V$ and  $V^*$.

\begin{proposition}\cite[Proposition 3.3]{HBG}
Let $(V,l,r)$ be a representation of a Novikov algebra $(A,\circ )$. Then $(V^*,l^*+r^*,-r^*)$ is also a representation of $(A,\circ).$
\end{proposition}

\begin{remark}
The adjoint representation $(A,L_{\circ},R_{\circ})$ of a Novikov algebra $(A,\circ)$ gives the representation $(A^*,L_{\circ}^*+R_{\circ}^*,-R_{\circ}^*)$.
\delete{ Let $(A,\circ )$ be a Novikov algebra and $(A,L_{\circ},R_{\circ})$ be the adjoint representation of $(A,\circ)$. Then by above Proposition, $(A^*,L_{\circ}^*+R_{\circ}^*,-R_{\circ}^*)$ is also a representation of $(A,\circ )$.}
\end{remark}

\begin{definition}\label{pNov}\cite{HBG}
Let $A$ be a vector space. A {\bf pre-Novikov algebra} is a triple $(A,\vartriangleleft,\vartriangleright)$, where $\vartriangleleft,\vartriangleright: A\otimes A\rightarrow A$ are binary operations satisfying
\begin{align}
a\vartriangleright (b\vartriangleright c)&=(a\circ b)\vartriangleright c+b\vartriangleright(a\vartriangleright c)-(b\circ a)\vartriangleright c,\\
a\vartriangleright(b\vartriangleleft c)&=(a\vartriangleright b)\vartriangleleft c+b\vartriangleleft(a\circ c)-(b\vartriangleleft a)\vartriangleleft c,\\
(a\circ b)\vartriangleright c&=(a\vartriangleright c)\vartriangleleft b,\\
(a\vartriangleleft b)\vartriangleleft c&=(a\vartriangleleft c)\vartriangleleft b,  \quad a,b,c\in A,
\end{align}
where $a\circ b=a\vartriangleleft b+a\vartriangleright b$.
\end{definition}

\begin{proposition}\cite[Proposition 3.33]{HBG}\label{an}
Let $(A,\vartriangleleft,\vartriangleright)$ be a pre-Novikov algebra. The binary operation
\begin{align}
a\circ b:=a\vartriangleleft b+a\vartriangleright b, \quad  a,b\in A, \label{e13}
\end{align}
defines a Novikov algebra, which is called the \textbf{associated Novikov algebra} of $(A,\vartriangleleft,\vartriangleright)$. Furthermore, $(A,L_{\vartriangleright},R_{\vartriangleleft})$ is a representation of $(A,\circ)$. Conversely, let $A$ be a vector space equipped with binary operations $\vartriangleleft$ and $\vartriangleright$. If $(A, \circ)$ defined by Eq. \eqref{e13} is a Novikov algebra and $(A,L_{\vartriangleright},R_{\vartriangleleft})$ is a representation of $(A,\circ)$, then $(A,\vartriangleleft,\vartriangleright)$ is a pre-Novikov algebra.
\end{proposition}

\begin{definition}\cite{HBG}
Let $(V,l,r)$ be a representation of a Novikov algebra $(A,\circ )$. A linear map $T:V\rightarrow A$ is called an \textbf{$\mathcal{O}$-operator} on $(A,\circ )$ associated to $(V,l,r)$ if $T$ satisfies
\begin{align}
T(u)\circ T(v)=T(l(T(u))v)+T(r(T(v))u),\quad  u,v\in V.
\end{align}
\end{definition}

\begin{proposition}\label{pro-o1}\cite[Proposition 3.34]{HBG}
Let $(V,l,r)$ be a representation of a Novikov algebra $(A,\circ )$. If $T$ is an $\mathcal{O}$-operator on $(A,\circ )$ associated to $(V,l,r)$, then there is a pre-Novikov algebra structure on $V$ defined by
\begin{align}
u\vartriangleright v:=l(T(u))v,\quad u\vartriangleleft v:=r(T(v))u, \quad  u,v\in V.
\end{align}
\end{proposition}

\begin{definition}\cite{HBG}
Let $(A,\circ )$ be a Novikov algebra. If there is a skewsymmetric nondegenerate bilinear form $\omega (\cdot,\cdot)$ on $A$ satisfying
\begin{align}
\omega (a\circ b,c)-\omega (a\circ c+c\circ a,b)+\omega (c\circ b,a)=0 ,\quad a,b,c\in A,\label{qn}
\end{align}
then $(A,\circ ,\omega(\cdot,\cdot))$ is called a \textbf{quasi-Frobenius Novikov algebra}.
\end{definition}

\delete{\begin{remark}\label{re1}
Let $V$ be a vector space. Any invertible linear map $T$: $V^*\rightarrow V$ can induce a nondegenerate bilinear form $\omega (\cdot,\cdot)$ on $V$ through
\begin{align}
\omega (u,v)=\langle T^{-1}u,v \rangle, \quad  u,v\in V.
\end{align}

Conversely, any nondegenerate bilinear form $\omega (\cdot,\cdot)$ on $V$ can induce an invertible linear map $T$: $V^*\rightarrow V$ defined by
\begin{align}
\omega (T(u^*),v)=\langle u^*,v\rangle, \quad  u^*\in V^*,v\in V. \label{e1}
\end{align}
\end{remark}}

Next, we show that there is a natural pre-Novikov algebra structure on $A$, when $(A,\circ ,\omega(\cdot,\cdot))$ is a quasi-Frobenius Novikov algebra.

\begin{theorem}\label{t1}
Let $(A,\circ ,\omega(\cdot,\cdot))$ be a quasi-Frobenius Novikov algebra. Then there is a compatible pre-Novikov algebra structure on $A$ given by
\begin{align}
\omega (a\vartriangleright b,c)&=\omega(a\circ c+c\circ a,b),\label{t11}\\
\omega(a\vartriangleleft b,c)&=\omega(a,c\circ b),\quad a,b,c\in A,\label{t12}
\end{align}
such that $(A,\circ )$ is the associated Novikov algebra of $(A,\vartriangleleft,\vartriangleright)$. This pre-Novikov algebra is called the \textbf{associated pre-Novikov algebra} of $(A,\circ ,\omega(\cdot,\cdot))$.
\end{theorem}

\begin{proof}
Obviously, the nondegenerate bilinear form $\omega(\cdot,\cdot)$ on $A$ can induce an invertible linear map $T: A^*\rightarrow A$ given by
\begin{eqnarray}
\label{eq1}\omega(T(f),a)=\langle f,a\rangle,\quad f\in A^*,a\in A.
\end{eqnarray}
Therefore, for all $a,b,c\in A$, we have
\begin{align*}
\omega (a\vartriangleright b,c)&=\omega(a\circ c+c\circ a,b)=-\omega (b,a\circ c+c\circ a)\\
&=-\langle T^{-1}(b), a\circ c+c\circ a \rangle =-\langle T^{-1}(b), (L_{\circ}+R_{\circ})(a)c \rangle\\
&=\langle (L_{\circ}^*+R_{\circ}^*)(a)T^{-1}(b),c \rangle=\omega(T((L_{\circ}^*+R_{\circ}^*)(a)T^{-1}(b)),c),
\end{align*}
and
\begin{align*}
\omega (a\vartriangleleft b,c)&=\omega (a,c\circ b)=\langle T^{-1}(a),c\circ b \rangle\\
&=\langle T^{-1}(a),R_{\circ}(b)c\rangle=-\langle R_{\circ}^*(b)T^{-1}(a),c\rangle\\
&=\omega(T((-R_{\circ}^*)(b)T^{-1}(a)),c).
\end{align*}
By the nondegenerate property of $\omega (\cdot,\cdot)$, we obtain
\begin{eqnarray}
a\vartriangleright b=T((L_{\circ}^*+R_{\circ}^*)(a)T^{-1}(b)),\;\;a\vartriangleleft b=T((-R_{\circ}^*)(b)T^{-1}(a)),\;\;a, b\in A.
\end{eqnarray}
Set $a=T(f),\ b=T(g)$. Define
\begin{align*}
f\vartriangleright_*g &:=(L_{\circ}^*+R_{\circ}^*)(T(f))g,\\
f\vartriangleleft_* g&:=(-R_{\circ}^*)(T(g))f,\quad f, g\in A^*.
\end{align*}
Then we have
\begin{align*}
a\vartriangleright b&=T(f)\vartriangleright T(g)=T(f\vartriangleright_* g),\\
a\vartriangleleft b&=T(f)\vartriangleleft T(g)=T(f\vartriangleleft_* g),\;\;a, b\in A.
\end{align*}
If we prove that $(A^\ast, \vartriangleleft_*, \vartriangleright_*)$ is a pre-Novikov algebra, then $(A, \lhd, \rhd)$ is a pre-Novikov algebra and $T$ is an isomorphism of pre-Novikov algebras. Therefore, we only need to show that $(A^\ast, \vartriangleleft_*,$ $ \vartriangleright_*)$ is a pre-Novikov algebra.

For all $f$, $g$, $h\in A^*$, we have
\begin{eqnarray*}
&&\langle h,T(f)\circ T(g)-T((L_{\circ}^*+R_{\circ}^*)(T(f))g-R_{\circ}^*(T(g))f)\rangle\\
&&\quad=\omega(T(h),T(f)\circ T(g))-\omega(T(h),T((L_{\circ}^*+R_{\circ}^*)(T(f))g-R_{\circ}^*(T(g))f))\\
&&\quad=\omega(T(h),T(f)\circ T(g))+\omega(T((L_{\circ}^*+R_{\circ}^*)(T(f))g-R_{\circ}^*(T(g))f),T(h))\\
&&\quad=\omega(T(h),T(f)\circ T(g))+\langle(L_{\circ}^*+R_{\circ}^*)(T(f))g-R_{\circ}^*(T(g))f, T(h)\rangle\\
&&\quad=\omega(T(h),T(f)\circ T(g))-\langle g,T(f)\circ T(h)+ T(h)\circ T(f)\rangle+\langle f ,T(h)\circ T(g)\rangle\\
&&\quad=\omega(T(h),T(f)\circ T(g))-\omega(T(g),T(f)\circ T(h)+ T(h)\circ T(f))+\omega(T(f) ,T(h)\circ T(g))\\
&&\quad=0.
\end{eqnarray*}
Therefore, we obtain
\begin{align*}
T(f)\circ T(g)-T((L_{\circ}^*+R_{\circ}^*)(T(f))g-R_{\circ}^*(T(g))f)=0.
\end{align*}
Hence $T:A^*\rightarrow A$ defined by Eq. \eqref{eq1} is an invertible $\mathcal{O}$-operator on $(A,\circ )$ associated to $(A^*,L_{\circ}^*+R_{\circ}^*,-R_{\circ}^*)$.
By Proposition \ref{pro-o1}, $(A^\ast, \vartriangleleft_*, \vartriangleright_*)$ is a pre-Novikov algebra.

Moreover,
\begin{eqnarray*}
a\lhd b+a\rhd b&=&T((L_{\circ}^*+R_{\circ}^*)(a)T^{-1}(b))+T((-R_{\circ}^*)(b)T^{-1}(a))\\
&=&T(f)\circ T(g)\\
&=&a\circ b.
\end{eqnarray*}
Therefore, $(A,\circ)$ is the associated Novikov algebra of $(A, \lhd, \rhd)$.

The proof is completed.
\end{proof}

\section{Double constructions of quasi-Frobenius Novikov algebras and pre-Novikov bialgebras}
In this section, we introduce the definitions of double constructions of quasi-Frobenius Novikov algebras and pre-Novikov bialgebras, and show that the double construction of quasi-Frobenius Novikov algebras is equivalent to a pre-Novikov bialgebra, which is also characterized by some matched pair of Novikov algebras.

First, we recall matched pairs of Novikov algebras.
\begin{proposition}\label{d1}\cite{H}
Let $(A,\circ)$ and $(B,\bullet)$  be Novikov algebras. If $(B,l_A,r_A)$ is a representation of $(A,\circ)$, $(A,l_B,r_B)$ is a representation of $(B,\bullet)$ and the following conditions are satisfied:
\begin{flalign}
&l_B(x)(a\circ b)=-l_B(l_A(a)x-r_A(a)x)b+(l_B(x)a-r_B(x)a)\circ b+r_B(r_A(b)x)a+a\circ (l_B(x)b),\label{e2}\\
&r_B(x)(a\circ b-b\circ a)=r_B(l_A(b)x)a-r_B(l_A(a)x)b+a\circ (r_B(x)b)-b\circ (r_B(x)a),\label{e3}\\
&l_A(a)(x\bullet y)=-l_A(l_B(x)a-r_B(x)a)y+(l_A(a)x-r_A(a)x)\bullet y+r_A(r_B(y)a)x+x\bullet(l_A(a)y),\label{e4}\\
&r_A(a)(x\bullet y-y\bullet x)=r_A(l_B(y)a)x-r_A(l_B(x)a)y+x\bullet (r_A(a)y)-y\bullet (r_A(a)x),\label{e5}\\
&(l_B(x)a)\circ b+l_B(r_A(a)x)b=(l_B(x)b)\circ a+l_B(r_A(b)x)a,\label{e6}\\
&(r_B(x)a)\circ b+l_B(l_A(a)x)b=r_B(x)(a\circ b),\label{e7}\\
&l_A(r_B(x)a)y+(l_A(a)x)\bullet y=l_A(r_B(y)a)x+(l_A(a)y)\bullet x,\label{e8}\\
&l_A(l_B(x)a)y+(r_A(a)x)\bullet y=r_A(a)(x\bullet y), \quad   a,b\in A,x,y\in B,\label{e9}
\end{flalign}
then there is a Novikov algebra structure on the direct sum $A\oplus B$ of the underlying vector spaces of $A$ and $B$ given by
\begin{flalign}
&(a+x)\cdot (b+y):=(a\circ b+l_B(x)b+r_B(y)a)+(x\bullet y+l_A(a)y+r_A(b)x),\;\;a, b\in A,\;\;x, y\in B.\label{e10}
\end{flalign}
 $(A,B,l_A,r_A,l_B,r_B)$ satisfying the above conditions is called a \textbf{matched pair of Novikov algebras.} Conversely, any Novikov algebra that can be decomposed into a linear direct sum of two Novikov subalgebras is obtained from a matched pair of Novikov algebras.
\end{proposition}

Next, we introduce the definition of double constructions of quasi-Frobenius Novikov algebras associated to pre-Novikov algebras.
\begin{definition}
Let $(A,\vartriangleleft,\vartriangleright)$, $(A^*,\vartriangleleft_\ast,\vartriangleright_\ast)$ be pre-Novikov algebras and $(A,\circ)$, $(A^\ast,\circ_\ast)$ be their associated Novikov algebras respectively. If a quasi-Frobenius Novikov algebra $(B, \cdot, \omega(\cdot,\cdot))$ satisfies the following conditions:
\begin{enumerate}
         \item $B$ is the direct sum of $A$ and $A^*$ as vector spaces,
         \item $(A,\vartriangleleft,\vartriangleright)$ and $(A^*,\vartriangleleft_\ast,\vartriangleright_\ast)$ are pre-Novikov subalgebras of $(B,\trianglelefteq,\trianglerighteq)$, which is the associated pre-Novikov algebra of $(B, \cdot, \omega(\cdot,\cdot))$,
         \item the bilinear form $\omega(\cdot,\cdot)$ on $B=A\oplus A^*$ is given by
\begin{flalign}
\omega(a+f,b+g)=\langle f,b\rangle -\langle g,a\rangle,\quad  a,b\in A,\;\;f, g\in A^*,\label{man}
\end{flalign}

       \end{enumerate}
then  $(B, \cdot, \omega(\cdot,\cdot))$ is called a {\bf double construction of quasi-Frobenius Novikov algebras} associated to $(A,\vartriangleleft,\vartriangleright)$ and $(A^*,\vartriangleleft_\ast,\vartriangleright_\ast)$.
\end{definition}

We give a characterization of double constructions of quasi-Frobenius Novikov algebras associated to pre-Novikov algebras by matched pairs of Novikov algebras.
\begin{proposition}\label{p1}
Let $(A,\vartriangleleft,\vartriangleright)$ be a pre-Novikov algebra and $(A,\circ )$ be the associated Novikov algebra of  $(A,\vartriangleleft,\vartriangleright)$. Suppose that there exists a pre-Novikov algebra structure $(A^*,\vartriangleleft_*,\vartriangleright_*)$ on the dual vector space $A^*$ and $(A^*,\circ_*)$ is the associated Novikov algebra of $(A^*,\vartriangleleft_*,\vartriangleright_*)$. Then there is a double construction of quasi-Frobenius Novikov algebras  associated to $(A,\vartriangleleft,\vartriangleright)$ and $(A^*,\vartriangleleft_\ast,\vartriangleright_\ast)$ if and only if $(A,A^*,L_{\vartriangleright}^*+R_{\vartriangleleft}^*,-R_{\vartriangleleft}^*,L_{\vartriangleright_*}^*+R_{\vartriangleleft_*}^*,-R_{\vartriangleleft_*}^*)$ is a matched pair of Novikov algebras.
\end{proposition}
\begin{proof}
Suppose that $(A,A^*,L_{\vartriangleright}^*+R_{\vartriangleleft}^*,-R_{\vartriangleleft}^*,L_{\vartriangleright_*}^*+R_{\vartriangleleft_*}^*,-R_{\vartriangleleft_*}^*)$ is a matched pair of Novikov algebras. Then by Proposition \ref{d1}, there is a Novikov algebra structure $(A\oplus A^\ast, \cdot)$ on the vector space $A\oplus A^*$ given by Eq. \eqref{e10}. By Eq. \eqref{man}, for all $a,b,c\in A$, and $f,g,h\in A^*$, we have
\begin{eqnarray*}
&&\omega((a+f)\cdot(b+g),c+h)\\
&&\;=\omega((a\circ b+(L_{\vartriangleright_*}^*+R_{\vartriangleleft_*}^*)(f)b+(-R_{\vartriangleleft_*}^*)(g)a)+(f\circ_* g+(L_{\vartriangleright}^*+R_{\vartriangleleft}^*)(a)g+(-R_{\vartriangleleft}^*)(b)f),c+h)\\
&&\;=\langle f\circ_* g+(L_{\vartriangleright}^*+R_{\vartriangleleft}^*)(a)g+(-R_{\vartriangleleft}^*)(b)f,c\rangle-\langle h,a\circ b+(L_{\vartriangleright_*}^*+R_{\vartriangleleft_*}^*)(f)b+(-R_{\vartriangleleft_*}^*)(g)a\rangle\\
&&\;=\langle f,c\vartriangleleft b\rangle-\langle g,a\vartriangleright c+c\vartriangleleft a \rangle-\langle h,a\circ b\rangle -\langle h\vartriangleleft_* g, a\rangle\\
&&\quad\;+\langle f\vartriangleright_*h+h\vartriangleleft_*f, b\rangle+\langle f\circ_* g, c\rangle.
\end{eqnarray*}
Similarly, we obtain
\begin{eqnarray*}
\omega((a+f)\cdot(c+h),b+g)&=&\langle f,b\vartriangleleft c\rangle-\langle g,a\circ c\rangle-\langle h,a\vartriangleright b+b\vartriangleleft a \rangle\\
&&-\langle g\vartriangleleft_* h, a\rangle+\langle f\circ_* h, b\rangle+\langle f\vartriangleright_*g+g\vartriangleleft_*f,c\rangle,\\
\omega((c+h)\cdot(a+f),b+g)&=&-\langle f,c\vartriangleright b+b\vartriangleleft c \rangle-\langle g,c\circ a\rangle+\langle h,b\vartriangleleft a\rangle\\
&&+\langle h\vartriangleright_*g+g\vartriangleleft_*h,a\rangle+\langle h\circ_* f,b\rangle-\langle g\vartriangleleft_* f,c\rangle,\\
\omega((c+h)\cdot(b+g),a+f)&=&-\langle f,c\circ b\rangle-\langle g,c\vartriangleright a+a\vartriangleleft c \rangle+\langle h,a\vartriangleleft b\rangle\\
&&+\langle h\circ_* g,a\rangle+\langle h\vartriangleright_*f+f\vartriangleleft_*h, b\rangle-\langle f\vartriangleleft_* g,c\rangle.
\end{eqnarray*}
Hence we get
\begin{eqnarray*}
&&\omega ((a+f)\cdot(b+g),c+h)-\omega ((a+f)\cdot(c+h)+(c+h)\cdot( a+f),b+g)\\
&&\quad\quad+\omega ((c+h)\cdot (b+g),a+f)\\
&&\;\;=\langle f,c\vartriangleleft b-b\vartriangleleft c+( c\vartriangleright b+b\vartriangleleft c) -c\circ b\rangle\\
&&\quad\quad+\langle g,-(a\vartriangleright c+c\vartriangleleft a)+a\circ c+c\circ a-(c\vartriangleright a+a\vartriangleleft c) \rangle  \\
&&\quad\quad+\langle h,-a\circ b+(a\vartriangleright b+b\vartriangleleft a)-b\vartriangleleft a+a\vartriangleleft b\rangle\\
&&\quad\quad+\langle a,-h\vartriangleleft_* g+g\vartriangleleft_* h-(h\vartriangleright_*g+g\vartriangleleft_*h)+h\circ_* g\rangle\\
&&\quad\quad+\langle b,(f\vartriangleright_*h+h\vartriangleleft_*f)-f\circ_* h-h\circ_* f+(h\vartriangleright_*f+f\vartriangleleft_*h)\rangle\\
&&\quad\quad+\langle c,f\circ_* g-(f\vartriangleright_*g+g\vartriangleleft_*f)+g\vartriangleleft_* f-f\vartriangleleft_* g\rangle\\
&&\;\;=0.
\end{eqnarray*}
Therefore, $(A\oplus A^*,\cdot,\omega(\cdot,\cdot))$ is a quasi-Frobenius Novikov algebra. Then by Theorem \ref{t1}, there is a compatible pre-Novikov algebra structure on $A\oplus A^*$. Denote it by $(A\oplus A^*,\trianglelefteq,\trianglerighteq)$. It suffices to check that $(A,\vartriangleleft,\vartriangleright)$ and $(A^*,\vartriangleleft_*,\vartriangleright_*)$ are pre-Novikov subalgebras of $(A\oplus A^*,\trianglelefteq,\trianglerighteq)$. For all $a,b,c\in A,h\in A^*$, we have
\begin{eqnarray*}
\omega(a\trianglerighteq b,c+h)-\omega(a\vartriangleright b,c+h)&=&\omega(a\cdot(c+h),b)+\omega((c+h)\cdot a,b)-\omega(a\vartriangleright b,c+h)\\
&=&-\langle h,a\vartriangleright b+b\vartriangleleft a\rangle+\langle h,b\vartriangleleft a\rangle+\langle h,a\vartriangleright b\rangle\\
&=&0,
\end{eqnarray*}
and
\begin{eqnarray*}
\omega(a\trianglelefteq b,c+h)-\omega(a\vartriangleleft b,c+h)&=&\omega(a,(c+h)\cdot b)-\omega(a\vartriangleleft b,c+h)\\
&=&-\langle h, a\vartriangleleft b\rangle+\langle h,a\vartriangleleft b \rangle\\
&=&0.
\end{eqnarray*}
By the nondegenerate property of $\omega(\cdot,\cdot)$, we obtain $a\trianglerighteq b=a\vartriangleright b$ and $a\trianglelefteq b=a\vartriangleleft b$ for all $a,b\in A$. Hence $(A,\vartriangleleft,\vartriangleright)$ is a pre-Novikov subalgebra of $(A\oplus A^*,\trianglelefteq,\trianglerighteq)$. Similarly, we can prove that $(A^*,\vartriangleleft_*,\vartriangleright_*)$ is a pre-Novikov subalgebra of $(A\oplus A^*,\trianglelefteq,\trianglerighteq)$.
Hence $(A\oplus A^*,\cdot,\omega(\cdot,\cdot))$ is a double construction of quasi-Frobenius Novikov algebras associated to $(A,\vartriangleleft,\vartriangleright)$ and $(A^*,\vartriangleleft_\ast,\vartriangleright_\ast)$.

Conversely, suppose that $(A\oplus A^*,\cdot,\omega(\cdot,\cdot))$ is a double construction of quasi-Frobenius Novikov algebras associated to $(A,\vartriangleleft,\vartriangleright)$ and $(A^*,\vartriangleleft_\ast,\vartriangleright_\ast)$. By Proposition \ref{d1}, the Novikov algebra $(A\oplus A^*,\cdot)$ is obtained from a matched pair of Novikov algebras $(A, A^\ast,l_A,r_A,l_{A^\ast},r_{A^\ast})$, where $\cdot$ is given by Eq. \eqref{e10}.
By Theorem \ref{t1}, we have
\begin{eqnarray*}
\langle l_A(a)f,b\rangle&=&\langle a\cdot f,b\rangle=\omega(b,a\cdot f)=\omega(b,a\cdot f+f\cdot a)-\omega(b,f\cdot a)\\
&=&\omega(f,a\vartriangleright b)+\omega(f,b\vartriangleleft a)\\
&=&\langle (L_{\vartriangleright}^*+R_{\vartriangleleft}^*)(a)f,b\rangle,\\
\end{eqnarray*}
and
\begin{eqnarray*}
\langle r_A(a)f,b\rangle&=&\langle f\cdot a,b\rangle=\omega(b,f\cdot a)\\
&=&\omega(b\vartriangleleft a,f)=\langle R_{\vartriangleleft}(a)b,f\rangle\\
&=&-\langle R_{\vartriangleleft}^*(a)f,b\rangle,\quad\quad \quad a,b\in A, f\in A^*.
\end{eqnarray*}
Therefore, we obtain $l_A=L_{\vartriangleright}^*+R_{\vartriangleleft}^*$ and $r_A=-R_{\vartriangleleft}^*$. Similarly, we can prove $l_{A^\ast}=L_{\vartriangleright_*}^*+R_{\vartriangleleft_*}^*$ and $r_{A^\ast}=-R_{\vartriangleleft_*}^*$.

The proof is completed.
\end{proof}

Next, we introduce the definition of pre-Novikov coalgebras.
\begin{definition}\label{co1}
A \textbf{pre-Novikov coalgebra} is a vector space $A$ with linear maps $\alpha, \beta:A\rightarrow A\otimes A$ satisfying
\begin{eqnarray}
&&\label{cob1}(\alpha\otimes \id)\alpha(a)+(\tau\otimes \id)(\id\otimes \alpha)\beta(a)-(\id\otimes(\alpha+\beta))\alpha(a)-(\tau\otimes \id)(\beta\otimes \id)\alpha(a)=0,\\
&&\label{cob2}(\id\otimes\beta)\beta(a)+(\tau\otimes \id)((\alpha+\beta)\otimes \id)\beta(a)-((\alpha+\beta)\otimes \id)\beta(a)-(\tau\otimes \id)(\id\otimes \beta)\beta(a)=0,\\
&&\label{cob3}(\id \otimes \tau)(\beta\otimes \id)\alpha(a)-((\alpha+\beta)\otimes \id)\beta(a)=0,\\
&&\label{cob4}(\id\otimes\tau)(\alpha\otimes \id)\alpha(a)-(\alpha\otimes \id)\alpha(a)=0,\quad  a\in A.
\end{eqnarray}
Denote it by $(A,\alpha,\beta)$.
\end{definition}

Let $V$ and $W$ be vector spaces. Suppose that $\phi:V\rightarrow W$ is a linear map. Then there is an induced dual linear map $\phi^*:W^*\rightarrow V^*$ defined by
\begin{align}
\langle \phi^*(f),v\rangle=\langle f, \phi(v)\rangle,\quad  v\in V,f\in W^*.\label{a1}
\end{align}
It is easy to see that
$(A,\alpha,\beta)$ is a pre-Novikov coalgebra if and only if $(A^*,\alpha^*,\beta^*)$ is a pre-Novikov algebra.

\begin{proposition}\label{p2}
Let $(A,\vartriangleleft,\vartriangleright)$ be a pre-Novikov algebra and $(A,\circ )$ be the associated Novikov algebra of $(A,\vartriangleleft,\vartriangleright)$. Suppose that there is a pre-Novikov algebra $(A^\ast, \vartriangleleft_*, \vartriangleright_*)$ which is induced from a pre-Novikov coalgebra $(A, \alpha, \beta)$, whose associated Novikov algebra is denoted by $(A^\ast, \circ_\ast)$. Then $(A,A^*,L_{\vartriangleright}^*+R_{\vartriangleleft}^*,-R_{\vartriangleleft}^*,L_{\vartriangleright_*}^*+R_{\vartriangleleft_*}^*,-R_{\vartriangleleft_*}^*)$ is a matched pair of Novikov algebras if and only if the following equations are satisfied:
\begin{align}
(\tau\alpha+\beta)(a\circ b)&=((L_{\vartriangleright}+2R_{\vartriangleleft})(a)\otimes \id+ \id\otimes L_{\circ}(a))(\tau\alpha+\beta)(b)\label{lfd1}\\
&\quad+(\id\otimes R_{\circ}(b))(2\tau\alpha+\beta)(a)-(R_{\vartriangleleft}(b)\otimes\id)\tau\alpha (a),\nonumber\\
\tau\alpha(a\circ b-b\circ a)&=((L_{\vartriangleright}+R_{\vartriangleleft})(a)\otimes\id+\id\otimes L_{\circ}(a))\tau\alpha (b)\label{lfd2}\\
&\quad-((L_{\vartriangleright}+R_{\vartriangleleft})(b)\otimes\id+\id\otimes L_{\circ}(b))\tau\alpha (a),\nonumber\\
(\alpha+\beta)(a\vartriangleright b+b\vartriangleleft a)&=(\id\otimes (R_{\vartriangleright}+L_{\vartriangleleft})(b))(2\tau\alpha+\beta)(a)-(L_{\vartriangleleft}(b)\otimes\id)\alpha (a)\label{lfd3}\\
&\quad+((L_{\vartriangleright}+2R_{\vartriangleleft})(a)\otimes\id+\id\otimes(L_{\vartriangleright}+R_{\vartriangleleft})(a))(\alpha+\beta)(b),\nonumber\\
(\alpha+\beta-\tau\alpha-\tau\beta)(b\vartriangleleft a)&=(\id\otimes L_{\vartriangleleft}(b))(\tau\alpha+\beta)(a)-( L_{\vartriangleleft}(b)\otimes \id)(\alpha+\tau\beta)(a)\label{lfd4}\\
&\quad+(\id\otimes R_{\vartriangleleft}(a))(\alpha+\beta)(b)-(R_{\vartriangleleft}(a)\otimes\id)(\tau\alpha+\tau\beta)(b),\nonumber\\
(\id\otimes R_{\circ}(b)-R_{\vartriangleleft}&(b)\otimes\id)(\tau\alpha+\beta)(a)=(\id\otimes R_{\circ}(a)-R_{\vartriangleleft}(a)\otimes\id)(\tau\alpha+\beta)(b),\label{lfd5}\\
\tau\alpha(a\circ b)&=(\id\otimes R_{\circ}(b))\tau\alpha(a)+((L_{\vartriangleright}+R_{\vartriangleleft})(a)\otimes\id)(\tau\alpha+\beta)(b),\label{lfd6}\\
(\id\otimes(R_{\vartriangleright}+L_{\vartriangleleft})(b))\tau\alpha(a)&=((R_{\vartriangleright}+L_{\vartriangleleft})(b)\otimes \id)\alpha(a)+(\id \otimes(L_{\vartriangleright}+R_{\vartriangleleft})(a))(\tau\alpha+\tau\beta)(b)\label{lfd7}\\
&\quad-((L_{\vartriangleright}+R_{\vartriangleleft})(a)\otimes\id)(\alpha+\beta)(b),\nonumber\\
(\alpha+\beta)(b\vartriangleleft a)&=(\id\otimes(R_{\vartriangleright}+L_{\vartriangleleft})(b))(\tau\alpha+\beta)(a)+(R_{\vartriangleleft}(a)\otimes\id)(\alpha+\beta)(b),\;\;a, b\in A.\label{lfd8}
\end{align}
\end{proposition}

\begin{proof}
By Proposition \ref{d1}, it suffices to show that Eqs. \eqref{lfd1}-\eqref{lfd8} are equivalent to Eqs. \eqref{e2}-\eqref{e9} respectively when $B=A^*$, $l_A=L_{\vartriangleright}^*+R_{\vartriangleleft}^*$, $r_A=-R_{\vartriangleleft}^*$,
$l_{A^\ast}=L_{\vartriangleright_*}^*+R_{\vartriangleleft_*}^*$ and $r_{A^{\ast}}=-R_{\vartriangleleft_*}^*$.

\delete{As an example we give an explicit proof that Eq. \eqref{lfd8} holds if and only if Eq. \eqref{e9} holds.}
Let $a\in A$ and $f$, $g\in A^\ast$. By Eq. \eqref{e9}, we have
\begin{align*}
- R_{\vartriangleleft}^*(a)(f\circ_* g)=&(L_{\vartriangleright}^*+R_{\vartriangleleft}^*)((L_{\vartriangleright_*}^*+R_{\vartriangleleft_*}^*)(f)a)g
-(R_{\vartriangleleft}^*(a)f)\circ_*g\\
=& (L_{\vartriangleright}^*+R_{\vartriangleleft}^*)((L_{\vartriangleright_*}^*+R_{\vartriangleleft_*}^*)(f)a)g
-(L_{\vartriangleright_*}+L_{\vartriangleleft_*})(R_{\vartriangleleft}^*(a)f)g.
\end{align*}
Let both sides of the above equation act on an arbitrary element $b\in A$. Then we have
$$-\langle R_{\vartriangleleft}^*(a)(f\circ_* g),b\rangle=\langle (L_{\vartriangleright}^*+R_{\vartriangleleft}^*)((L_{\vartriangleright_*}^*+R_{\vartriangleleft_*}^*)(f)a)g,b\rangle
-\langle(L_{\vartriangleright_*}+L_{\vartriangleleft_*})(R_{\vartriangleleft}^*(a)f)g,b\rangle,$$
which is equivalent to the following equation
\begin{eqnarray}
\label{eq-q1}\langle f\otimes g,(\alpha+\beta)(b\vartriangleleft a)\rangle&=&\langle f\otimes g,(\id\otimes(R_{\vartriangleright}+L_{\vartriangleleft})(b))(\tau\alpha+\beta)(a)\rangle\\
&&\quad+\langle f\otimes g, (R_{\vartriangleleft}(a)\otimes\id)(\alpha+\beta)(b)\rangle.\nonumber
\end{eqnarray}
Then it is easy to see that Eq. \eqref{lfd8} holds if and only if Eq. \eqref{e9} holds. The others can be proved similarly.
\end{proof}

\begin{definition}
Let $(A, \lhd, \rhd)$ be a pre-Novikov algebra and $(A,\alpha,\beta)$ be a pre-Novikov coalgebra. If they also satisfy Eqs. \eqref{lfd1}-\eqref{lfd8}, then we call $(A,\vartriangleleft,\vartriangleright,\alpha,\beta)$ a \textbf{pre-Novikov bialgebra}.
\end{definition}

By Propositions \ref{p1} and \ref{p2}, we obtain the following conclusion.

\begin{theorem}\label{t4}
Let $(A,\vartriangleleft,\vartriangleright)$ be a pre-Novikov algebra and $(A,\circ )$ be the associated Novikov algebra of $(A,\vartriangleleft,\vartriangleright)$. Suppose that there is a pre-Novikov algebra $(A^\ast, \vartriangleleft_*, \vartriangleright_*)$ which is induced from a pre-Novikov coalgebra $(A, \alpha, \beta)$, whose associated Novikov algebra is denoted by $(A^\ast, \circ_\ast)$. Then the following conditions are equivalent.
\begin{enumerate}
\item There is a double construction of quasi-Frobenius Novikov algebras associated to $(A, \vartriangleleft,$ $\vartriangleright)$ and $(A^\ast, \vartriangleleft_\ast,\vartriangleright_\ast)$;
\item $(A,A^*,L_{\vartriangleright}^*+R_{\vartriangleleft}^*,-R_{\vartriangleleft}^*,L_{\vartriangleright_*}^*+R_{\vartriangleleft_*}^*,-R_{\vartriangleleft_*}^*)$ is a matched pair of Novikov algebras;
 \item $(A,\vartriangleleft,\vartriangleright,\alpha,\beta)$ is a pre-Novikov bialgebra.
 \end{enumerate}
\end{theorem}

Finally, we present an example of pre-Novikov bialgebras and quasi-Frobenius Novikov algebras.
\begin{example}\label{ex1}
Let $(A=\mathbf{k}e_1\oplus \mathbf{k}e_2,\vartriangleleft,\vartriangleright)$ be a two-dimensional vector space with binary operations
$\vartriangleleft,\vartriangleright$ given by
\begin{flalign*}
&e_i\vartriangleright e_j=0, \;\;i,j\in\{1,2\},\\
e_1\vartriangleleft e_1=e_1,\quad &e_1\vartriangleleft e_2=e_2,\quad e_2\vartriangleleft e_1=e_2,\quad e_2\vartriangleleft e_2 =0.
\end{flalign*}
One can directly check that $(A,\vartriangleleft,\vartriangleright)$ is a pre-Novikov algebra. Define linear maps $\alpha,\beta:A\rightarrow A\otimes A$ by
\begin{flalign*}
&\alpha(e_1)=e_2\otimes e_2, \quad \ \ \alpha(e_2)=0,\\
&\beta(e_1)=-e_2\otimes e_2, \quad  \beta(e_2)=0.
\end{flalign*}
Then it is easy to check that $(A,\vartriangleleft,\vartriangleright,\alpha,\beta)$ is a pre-Novikov bialgebra.

Let $\{e_1^*,e_2^*\}$ be the dual basis of $A^\ast$. Then by Theorem \ref{t4}, there is a double construction of quasi-Frobenius Novikov
algebras $(A\oplus A^*, \cdot, \omega(\cdot,\cdot))$ which is defined by non-zero products
\begin{eqnarray*}
&&e_1\cdot e_1=e_1,\;\;e_1\cdot e_2=e_2,\;\; e_2\cdot e_1=e_2,\;\;
e_1\cdot e_1^*=-e_1^*,\\
&&e_1^*\cdot e_1=e_1^*,\;\;e_1\cdot e_2^*=e_2-e_2^*,\;e_2^*\cdot e_1=e_2^*,\;e_2\cdot e_2^*=-e_1^*,\;e_2^*\cdot e_2=e_1^*,
\end{eqnarray*}
and the nondegenerate bilinear form $\omega (\cdot,\cdot)$ given by
\begin{flalign*}
&\omega (e_1^*,e_1)=\omega (e_2^*,e_2)=-\omega (e_1,e_1^*)=-\omega (e_2,e_2^*)=1,\\
&\omega (e_1,e_2^*)=\omega (e_2,e_1^*)=\omega (e_1^*,e_2)=\omega (e_2^*,e_1)=\omega(e_i,e_j)=\omega(e_i^*,e_j^*)=0,\;\;i, j\in \{1,2\}.
\end{flalign*}
\end{example}

\section{Pre-Novikov Yang-Baxter equation}
In this section, we introduce the definition of pre-Novikov Yang-Baxter equation whose symmetric solutions can be used to construct pre-Novikov bialgebras. Moreover, the operator forms of pre-Novikov Yang-Baxter equation are investigated.

Let $(A,\vartriangleleft,\vartriangleright)$ be a pre-Novikov algebra and $(A,\circ)$ be the associated Novikov algebra. For convenience, we define two binary operations $\odot$ and $\star$ on $A$ by
$$a\odot b:=a\vartriangleright b+b\vartriangleleft a,\qquad a\star b:=a\circ b+b\circ a, \quad  a,b \in A.$$
\begin{lemma}\label{l1}
Let $(A,\vartriangleleft,\vartriangleright)$ be a pre-Novikov algebra and $r\in A\otimes A$. Suppose that $\alpha,\beta:A\rightarrow A\otimes A$ are linear maps defined by
\begin{flalign}
\alpha(a):&=(L_{\circ}(a)\otimes \id+\id \otimes(L_{\vartriangleright}+R_{\vartriangleleft})(a))\tau r,\label{e11}\\
\beta(a):&=-(L_{\vartriangleright}(a)\otimes \id+\id \otimes (L_{\circ}+R_{\circ})(a))r, \quad a\in A.\label{e12}
\end{flalign}
Then $\alpha,\beta$ satisfy Eqs. \eqref{lfd1}-\eqref{lfd8} if and only if the following conditions are satisfied:
\begin{align}
&\big(((L_{\vartriangleright}+2R_{\vartriangleleft})(a)\otimes \id+\id\otimes (L_{\vartriangleright}+R_{\vartriangleleft})(a))(L_{\circ}(b)\otimes \id+\id\otimes (L_{\vartriangleright}+R_{\vartriangleleft})(b))\label{ld1}\\
&\;\;-(L_{\vartriangleleft}(b)\otimes \id)
(L_{\circ}(a)\otimes \id+\id\otimes(L_{\vartriangleright}+R_{\vartriangleleft})(a))\nonumber\\
&\;\;-(L_{\circ}(a\vartriangleright b+b\vartriangleleft a)\otimes \id+\id\otimes(L_{\vartriangleright}+R_{\vartriangleleft})(a\vartriangleright b+b\vartriangleleft a))\big)(\tau r-r)=0,\nonumber\\
&\big((\id\otimes R_{\vartriangleleft}(a))(L_{\circ}(b)\otimes \id+\id\otimes
(L_{\vartriangleright}+R_{\vartriangleleft})(b))
+(R_{\vartriangleleft}(a)\otimes \id)((L_{\circ}+R_{\circ})(b)\otimes \id\label{ld2}\\
&\;\;+\id\otimes L_{\vartriangleright}(b))
-(L_{\vartriangleleft}(b)\otimes \id)
(\id\otimes R_{\vartriangleleft}(a)-R_{\circ}(a)\otimes \id)\nonumber\\
&\;\;-(\id\otimes(2L_{\vartriangleright}
+R_{\vartriangleleft})(b\vartriangleleft a)+(2 L_{\circ}+R_{\circ})(b\vartriangleleft a)\otimes \id)\big)(\tau r-r)=0,\nonumber\\
&\big(((R_{\vartriangleright}+L_{\vartriangleleft})(b)\otimes \id)(L_{\circ}(a)\otimes \id+\id\otimes(L_{\vartriangleright}+R_{\vartriangleleft})(a))-
(\id \otimes(L_{\vartriangleright}+R_{\vartriangleleft})(a))(\id\otimes L_{\vartriangleright}(b)\label{ld3}\\
&\;\;+(L_{\circ}+R_{\circ})(b)\otimes \id)
-((L_{\vartriangleright}+R_{\vartriangleleft})(a)\otimes \id)(L_{\circ}(b)\otimes \id+\id\otimes(L_{\vartriangleright}+R_{\vartriangleleft})(b))\big)(\tau r-r)=0,\nonumber\\
&\big((R_{\vartriangleleft}(a)\otimes \id)(L_{\circ}(b)\otimes \id+\id\otimes(L_{\vartriangleright}+R_{\vartriangleleft})(b))\label{ld4}\\
&\;\;-(L_{\circ}(b\vartriangleleft a)\otimes \id+\id\otimes(L_{\vartriangleright}+R_{\vartriangleleft})(b\vartriangleleft a))\big)(\tau r-r)=0,\;\;\;a,b\in A.\nonumber
\end{align}
\end{lemma}

\begin{proof}
It is easy to check that $\alpha,\beta$ satisfy Eqs. \eqref{lfd1}, \eqref{lfd2}, \eqref{lfd5} and \eqref{lfd6} automatically. \delete{In fact, we only need to check the following correspondences:
\begin{center}
 Eq. \eqref{lfd3} $\Leftrightarrow $ Eq. \eqref{ld1}, \qquad Eq. \eqref{lfd4} $\Leftrightarrow$ Eq. \eqref{ld2},\\
 Eq. \eqref{lfd7} $\Leftrightarrow $ Eq. \eqref{ld3}, \qquad Eq. \eqref{lfd8} $\Leftrightarrow$ Eq. \eqref{ld4}.
\end{center}}

By Eq. \eqref{lfd8}, we have
\begin{eqnarray*}
&&(\id\otimes(R_{\vartriangleright}+L_{\vartriangleleft})(b))(\tau\alpha+\beta)(a)+(R_{\vartriangleleft}(a)\otimes \id)(\alpha+\beta)(b)-(\alpha+\beta)(b\vartriangleleft a)\\
&&\;\;=(\id\otimes(R_{\vartriangleright}+L_{\vartriangleleft})(b))(R_{\vartriangleleft}(a)\otimes \id-\id\otimes R_{\circ}(a))r\\
&&\;\;\quad+(R_{\vartriangleleft}(a)\otimes \id)(L_{\circ}(b)\otimes \id+\id\otimes (L_{\vartriangleright}+R_{\vartriangleleft})(b))\tau r\\
&&\;\;\quad-(R_{\vartriangleleft}(a)\otimes \id)(L_{\vartriangleright}(b)\otimes \id+\id\otimes(L_{\circ}+R_{\circ})(b))r\\
&&\;\;\quad-(L_{\circ}(b  \vartriangleleft a)\otimes \id+\id\otimes(L_{\vartriangleright}+R_{\vartriangleleft})(b  \vartriangleleft a) )\tau r\\
&&\;\;\quad+(L_{\vartriangleright}(b \vartriangleleft a)\otimes \id+\id\otimes(L_{\circ}+R_{\circ})(b  \vartriangleleft a))r\\
&&\;\;=((R_{\vartriangleleft}(a)\otimes \id)(L_{\circ}(b)\otimes \id+\id\otimes(L_{\vartriangleright}+R_{\vartriangleleft})(b))\\
&&\;\;\quad-L_{\circ}(b\vartriangleleft a)\otimes \id-\id\otimes(L_{\vartriangleright}+R_{\vartriangleleft})(b\vartriangleleft a))\tau r\\
&&\;\;\quad+(\id\otimes(-(R_{\vartriangleright}+L_{\vartriangleleft})(b)R_{\circ}(a)+(L_{\circ}+R_{\circ})(b \vartriangleleft  a))\\
&&\;\;\quad+(-R_{\vartriangleleft}(a)L_{\vartriangleright}(b)+L_{\vartriangleright}(b  \vartriangleleft a))\otimes \id-R_{\vartriangleleft}(a)\otimes( L_{\vartriangleright}+R_{\vartriangleleft})(b) )r.
\end{eqnarray*}
By Definition \ref{pNov}, for all $a,b,c\in A$, we have
\begin{eqnarray*}
&&-(c\circ a)\odot b+c\star (b\vartriangleleft a)=(b\vartriangleleft a)\odot c,
\end{eqnarray*}
and
\begin{eqnarray*}
&&-(b \vartriangleright c)\vartriangleleft a+(b\vartriangleleft a) \vartriangleright c=-(b\vartriangleright a)\vartriangleright c.
\end{eqnarray*}
Hence we obtain
\begin{eqnarray*}
&&-(R_{\vartriangleright}+L_{\vartriangleleft})(b)R_{\circ}(a)+(L_{\circ}+R_{\circ})(b \vartriangleleft  a)=(L_{\vartriangleright}+R_{\vartriangleleft})(b\vartriangleleft a),
\end{eqnarray*}
and
\begin{eqnarray*}
&&-R_{\vartriangleleft}(a)L_{\vartriangleright}(b)+L_{\vartriangleright}(b  \vartriangleleft a)=-L_{\vartriangleright}(b\vartriangleleft a).
\end{eqnarray*}
Therefore, one gets
\begin{eqnarray*}
&&(\id\otimes(R_{\vartriangleright}+L_{\vartriangleleft})(b))(\tau\alpha+\beta)(a)+(R_{\vartriangleleft}(a)\otimes \id)(\alpha+\beta)(b)-(\alpha+\beta)(b\vartriangleleft a)\\
&&\;\;=((R_{\vartriangleleft}(a)\otimes \id)(L_{\circ}(b)\otimes \id+\id\otimes(L_{\vartriangleright}+R_{\vartriangleleft})(b))
-L_{\circ}(b\vartriangleleft a)\otimes \id \\
&&\;\;\quad-\id\otimes(L_{\vartriangleright}+R_{\vartriangleleft})(b\vartriangleleft a))\tau r
+(\id\otimes (L_{\vartriangleright}+R_{\vartriangleleft})(b\vartriangleleft a)\\
&&\;\;\quad-R_{\vartriangleleft}(a)\otimes (L_{\vartriangleright}+R_{\vartriangleleft})(b)-L_{\vartriangleright}(b\vartriangleright a)\otimes \id )r\\
&&\;\;=((R_{\vartriangleleft}(a)\otimes \id)(L_{\circ}(b)\otimes \id+\id\otimes(L_{\vartriangleright}+R_{\vartriangleleft})(b))\\
&&\;\;\quad-(L_{\circ}(b\vartriangleleft a)\otimes \id+\id\otimes(L_{\vartriangleright}+R_{\vartriangleleft})(b\vartriangleleft a)))(\tau r-r)\\
&&\;\;\quad+(R_{\vartriangleleft}(a)L_{\circ}(b)\otimes \id+R_{\vartriangleleft}(a)\otimes (L_{\vartriangleright}+R_{\vartriangleleft})(b)
-L_{\circ}(b\vartriangleleft a)\otimes \id \\
&&\;\;\quad-\id\otimes(L_{\vartriangleright}+R_{\vartriangleleft})(b\vartriangleleft a)
+\id\otimes (L_{\vartriangleright}+R_{\vartriangleleft})(b\vartriangleleft a)\\
&&\;\;\quad-R_{\vartriangleleft}(a)\otimes (L_{\vartriangleright}+R_{\vartriangleleft})(b)-L_{\vartriangleright}(b\vartriangleright a)\otimes \id )r\\
&&\;\;=(R_{\vartriangleleft}(a)\otimes \id)(L_{\circ}(b)\otimes \id+\id\otimes(L_{\vartriangleright}+R_{\vartriangleleft})(b))\\
&&\;\;\quad-(L_{\circ}(b\vartriangleleft a)\otimes \id+\id\otimes(L_{\vartriangleright}+R_{\vartriangleleft})(b\vartriangleleft a))(\tau r-r)\\
&&\;\;\quad+(R_{\vartriangleleft}(a)
L_{\circ}(b)\otimes \id-L_{\circ}(b\vartriangleleft a)\otimes \id-L_{\vartriangleright}(b\vartriangleright a)\otimes \id)r.
\end{eqnarray*}
By Definition \ref{pNov}, for all $a,b,c\in A$, we have
\begin{eqnarray*}
&&(b\circ a)\vartriangleleft c-(b\vartriangleleft a)\circ c-(b\vartriangleright a)\vartriangleright c\\
&&\;\;=(b\vartriangleright c)\vartriangleleft a+(b\vartriangleleft c)\vartriangleleft a-(b\vartriangleleft a)\vartriangleleft c-(b\vartriangleleft a)\vartriangleright c-(b\vartriangleright a)\vartriangleright c\\
&&\;\;=((b\vartriangleright c)\vartriangleleft a-(b\vartriangleleft a)\vartriangleright c-(b\vartriangleright a)\vartriangleright c)+((b\vartriangleleft c)\vartriangleleft a-(b\vartriangleleft a)\vartriangleleft c)\\
&&\;\;=0.
\end{eqnarray*}
Hence we obtain
$$R_{\vartriangleleft}(a)
L_{\circ}(b)\otimes \id-L_{\circ}(b\vartriangleleft a)\otimes \id -L_{\vartriangleright}(b\vartriangleright a)\otimes \id=0.$$
Therefore  Eq. \eqref{lfd8} is equivalent to Eq. \eqref{ld4}.

One similarly verifies that  Eq. \eqref{lfd3} is equivalent to Eq. \eqref{ld1},  Eq. \eqref{lfd4} is equivalent to Eq. \eqref{ld2}, and Eq. \eqref{lfd7} is equivalent to Eq. \eqref{ld3}. Then the proof is completed.
\end{proof}

Let $V$ be a vector space with a binary operation $\ast$. Let $r=\sum\limits_{i}x_i\otimes y_i \in V\otimes V$ and $r^{'}=\sum\limits_{i}x_i^{'}\otimes y_i^{'} \in V\otimes V$. Set
\begin{eqnarray*}
r_{12}\ast r_{13}^{'}:=\sum_{i,j}x_i\ast x_j^{'}\otimes y_i\otimes y_j^{'},\;r_{12}\ast r_{23}^{'}:=\sum_{i,j}x_i\otimes y_i\ast x_j^{'} \otimes y_j^{'},\;r_{13}\ast r_{12}^{'}:=\sum_{i,j}x_i\ast x_j^{'}\otimes y_j^{'}\otimes y_i,\\
r_{13}\ast r_{21}^{'}:=\sum_{i,j}x_i\ast y_j^{'}\otimes x_j^{'}\otimes y_i,\;r_{13}\ast r_{23}^{'}:=\sum_{i,j}x_i\otimes x_j^{'}\otimes y_i\ast y_j^{'},\; r_{21}\ast r_{13}^{'}:=\sum_{i,j}y_i\ast x_j^{'}\otimes x_i\otimes y_j^{'},\\
r_{21}\ast r_{23}:=\sum_{i,j}y_i\otimes x_i\ast x_j^{'}\otimes y_j^{'},\; r_{21}\ast r_{31}^{'}:=\sum_{i,j}y_i\ast y_j^{'}\otimes x_i\otimes x_j^{'},\; r_{21}\ast r_{32}^{'}:=\sum_{i,j}y_i\otimes x_i\ast y_j^{'}\otimes x_j^{'},
\end{eqnarray*}
\begin{eqnarray*}
r_{31}\ast r_{21}^{'}:=\sum_{i,j}y_i\ast y_j^{'}\otimes x_j^{'}\otimes x_i,\;r_{31}\ast r_{23}^{'}:=\sum_{i,j}y_i\otimes x_j^{'}\otimes x_i\ast y_j^{'},\;
r_{31}\ast r_{32}^{'}:=\sum_{i,j}y_i\otimes y_j^{'}\otimes x_i\ast x_j^{'},\\
r_{23}\ast r_{12}^{'}:=\sum_{i,j}x_j^{'}\otimes x_i\ast y_j^{'}\otimes y_i,\; r_{23}\ast r_{21}^{'}:=\sum_{i,j}y_j^{'}\otimes x_i\ast x_j^{'}\otimes y_i,\;
r_{23}\ast r_{13}^{'}:=\sum_{i,j}x_j^{'}\otimes x_i\otimes y_i\ast y_j^{'},\\
r_{23}\ast r_{31}^{'}:=\sum_{i,j}y_j^{'}\otimes x_i\otimes y_i\ast x_j^{'},\;\; r_{32}\ast r_{21}:=\sum_{i,j}y_j^{'}\otimes y_i\ast x_j^{'}\otimes x_i.
\end{eqnarray*}


\begin{lemma}\label{l3}
Let $(A,\vartriangleleft,\vartriangleright)$ be a pre-Novikov algebra and $r=\sum\limits_{i}x_i \otimes y_i \in A\otimes A$. Let $\alpha,\beta:A\rightarrow A\otimes A$ be linear maps defined by Eqs. \eqref{e11} and \eqref{e12} respectively. Then $(A, \alpha, \beta)$ is a pre-Novikov coalgebra if and only if the following conditions are satisfied for all $a\in A$:
\begin{eqnarray}
&&(L_{\circ}(a)\otimes \id \otimes \id)R_{11}+(\id \otimes L_{\vartriangleright}(a)\otimes \id)R_{12}+(\id\otimes \id \otimes L_{\odot}(a))R_{13}\label{r1}\\
&&\quad -\sum_{j}\big((\id\otimes L_{\vartriangleright}(a\odot x_{j})\otimes \id)(y_j\otimes(\tau r-r))\big)\nonumber\\
&&\quad-\sum_{i}\big((\id \otimes \id\otimes L_{\odot}(a\odot x_{i}))(y_i\otimes(\tau r-r))\big)=0,\nonumber\\
&&(L_{\vartriangleright}(a)\otimes \id\otimes \id-\id\otimes L_{\vartriangleright}(a)\otimes \id)R_{21}+(\id\otimes \id\otimes L_{\star}(a))R_{22}\label{r2}\\
&&\quad+\sum_{j}\big(((2L_{\vartriangleright}+R_{\vartriangleleft})(a\vartriangleright x_j)\otimes \id\otimes \id)((\tau r-r)\otimes y_j)\nonumber\\
&&\quad+(\id \otimes(2L_{\vartriangleright}+R_{\vartriangleleft})(a\vartriangleright x_j)\otimes \id)((\tau r-r)\otimes y_j)\big)=0,\nonumber\\
&&-(\id\otimes L_{\odot}(a)\otimes \id)R_{21}+(\id \otimes \id\otimes L_{\star}(a))R_{31}\label{r3}\\
&&\quad+\sum_{j}\big((L_{\vartriangleright}(a\vartriangleright x_j)\otimes \id\otimes \id)((\tau r-r)\otimes y_j)+(\id\otimes L_{\odot}(a\vartriangleright x_j)\otimes \id)((\tau r-r)\otimes y_j)\big)=0,\nonumber\\
&&-(\id\otimes L_{\odot}(a)\otimes \id)R_{12}+(\id\otimes \id\otimes L_{\odot}(a))R_{41}=0,\label{r4}
\end{eqnarray}
where
\begin{flalign*}
R_{11}&=r_{21}\circ r_{31}-r_{21}\circ r_{32}-r_{31}\odot r_{32}+r_{21} \vartriangleright  r_{23}+r_{31}\star r_{23},\\
R_{12}&=-r_{21}\circ r_{31}-r_{23}\odot r_{31}+r_{21}\vartriangleleft r_{32},\\
R_{13}&=r_{31}\circ r_{21}+r_{32}\odot r_{21}+r_{23}\vartriangleright r_{31}-r_{31}\vartriangleleft r_{23}+r_{32}\vartriangleright r_{21}+r_{31}\star r_{21},\\
R_{21}&=r_{21}\vartriangleright r_{13}+r_{12}\vartriangleright r_{23}+r_{13}\star r_{23},\\
R_{22}&=r_{13}\circ r_{21}+r_{23}\odot r_{21}-r_{13}\vartriangleright r_{12}-r_{23}\star r_{12}-r_{23}\circ r_{12}\\
&\quad -r_{13}\odot r_{12}+r_{23}\vartriangleright r_{21}+r_{13}\star r_{21}+r_{23}\circ r_{13}-r_{13}\circ r_{23},\\
R_{31}&=-r_{13} \circ r_{23}+r_{13} \circ  r_{21}+r_{23} \odot  r_{21}-r_{13} \vartriangleright  r_{12}-r_{23} \star  r_{12},\\
R_{41}&=-r_{31}\circ r_{21}-r_{32}\odot r_{21}+r_{31}\vartriangleleft r_{23}.
\end{flalign*}
\end{lemma}
\begin{proof}
Let $a\in A$. By Eqs. \eqref{e11} and \eqref{e12}, we have
\begin{flalign*}
A(a)&=(\alpha\otimes \id)\alpha(a)+(\tau\otimes \id)(\id\otimes \alpha)\beta(a)-(\id\otimes(\alpha+\beta))\alpha(a)-(\tau\otimes \id)(\beta\otimes \id)\alpha(a)\\
&=\sum\limits_{i,j}\big((a\circ y_i)\circ y_j\otimes x_j\otimes x_i+y_j\otimes(a\circ y_i)\odot x_j\otimes x_i+y_i\circ y_j\otimes x_j \otimes a\odot x_i\\
&\quad+y_j\otimes y_i\odot x_j\otimes a\odot x_i-y_i\circ y_j\otimes a\vartriangleright x_i\otimes x_j-y_j\otimes a\vartriangleright x_i\otimes y_i\odot x_j\\
&\quad-(a\star y_i)\circ y_j\otimes x_i\otimes x_j-y_j\otimes x_i\otimes(a\star y_i)\odot x_j-a\circ y_i\otimes x_i\circ y_j\otimes x_j\\
&\quad-a\circ y_i\otimes y_j\otimes x_i\odot x_j+a\circ y_i\otimes x_i\vartriangleright x_j\otimes y_j+a\circ y_i\otimes x_j\otimes x_i\star y_j\\
&\quad-y_i\otimes(a\odot x_i)\circ y_j\otimes x_j-y_i\otimes y_j\otimes(a\odot x_i)\odot x_j+y_i\otimes(a\odot x_i)\vartriangleright x_j\otimes y_j\\
&\quad+y_i\otimes x_j\otimes (a\odot x_i)\star y_j+y_j\otimes(a\circ y_i)\vartriangleright x_j\otimes x_i+(a\circ y_i)\star y_j\otimes x_j\otimes x_i\\
&\quad+y_j\otimes y_i\vartriangleright x_j\otimes a\odot x_i+y_i\star y_j\otimes x_j\otimes a\odot x_i\big).
\end{flalign*}
Then we obtain $A(a)=P_1(a)+P_2(a)+P_3(a)$, where\\
\begin{flalign*}
P_1(a)&=\sum\limits_{i,j}\big((a\circ y_i)\circ y_j\otimes x_j\otimes x_i
-(a\star y_i)\circ y_j\otimes x_i\otimes x_j\\
&\quad-a\circ y_i\otimes x_i\circ y_j\otimes x_j
-a\circ y_i\otimes y_j\otimes x_i\odot x_j\\
&\quad+a\circ y_i\otimes x_i\vartriangleright x_j\otimes y_j
+a\circ y_i\otimes x_j\otimes  x_i\star y_j
+(a\circ y_i)\star y_j\otimes x_j\otimes x_i\big),\\
P_2(a)&=\sum\limits_{i,j}\big(y_j\otimes(a\circ y_i)\odot x_j\otimes x_i-y_i\circ y_j\otimes a\vartriangleright x_i\otimes x_j-y_j\otimes a\vartriangleright x_i\otimes y_i\odot x_j\\
&\quad-y_i\otimes(a\odot x_i)\circ y_j\otimes x_j+y_i\otimes (a\odot x_i)\vartriangleright x_j\otimes y_j+y_j\otimes(a\circ y_i)\vartriangleright x_j\otimes x_i\big),\\
P_3(a)&=\sum\limits_{i,j}\big(y_i\circ y_j\otimes x_j \otimes a\odot x_i+y_j\otimes y_i\odot x_j\otimes a\odot x_i\\
&\quad-y_j\otimes x_i\otimes(a\star y_i)\odot x_j-y_i\otimes y_j\otimes(a\odot x_i)\odot x_j\\
&\quad+y_i\otimes x_j\otimes (a\odot x_i)\star y_j+y_j\otimes y_i\vartriangleright x_j\otimes a\odot x_i+y_i\star y_j\otimes x_j\otimes a\odot x_i\big).
\end{flalign*}
For $P_1(a)$, since
\begin{flalign*}
&\sum\limits_{i,j}\big((a\circ y_i)\circ y_j\otimes x_j\otimes x_i-(a\star y_i)\circ y_j\otimes x_i\otimes x_j+(a\circ y_i)\star y_j\otimes x_j\otimes x_i\big)\\
&\;\;=\sum\limits_{i,j}a\circ(y_j\circ y_i)\otimes x_j\otimes x_i,
\end{flalign*}
we have
\begin{flalign*}
P_1(a)&=\sum\limits_{i,j}\big((L_{\circ}\otimes \id\otimes \id)(y_j\circ y_i\otimes x_j\otimes x_i- y_i\otimes x_i\circ y_j\otimes x_j-y_i\otimes y_j\otimes x_i\odot x_j\\
&\quad +y_i\otimes x_i\vartriangleright x_j\otimes y_j+y_i\otimes x_j\otimes x_i\star y_j)\big)\\
&=(L_{\circ}\otimes \id\otimes \id)(r_{21}\circ r_{31}-r_{21}\circ r_{32}-r_{31}\odot r_{32}+r_{21} \vartriangleright  r_{23}+r_{31}\star r_{23}).
\end{flalign*}
For $P_2(a)$, since
\begin{eqnarray*}
&&\sum\limits_{i,j}\big(y_j\otimes(a\circ y_i)\odot x_j\otimes x_i-y_i\otimes(a\odot x_i)\circ y_j\otimes x_j
+y_i\otimes (a\odot x_i)\vartriangleright x_j\otimes y_j\\
&&\;\;\quad+y_j\otimes(a\circ y_i)\vartriangleright x_j\otimes x_i\big)\\
&&\;\;=\sum\limits_{i,j}y_j \otimes a\vartriangleright(x_j\vartriangleleft y_i)\otimes x_i-\sum\limits_{j}(\id\otimes L_{\vartriangleright}(a\odot x_{j})\otimes \id)(y_j\otimes(\tau r-r)),
\end{eqnarray*}
we have
\begin{flalign*}
P_2(a)&=\sum\limits_{i,j}\big((\id\otimes L_{\vartriangleright}(a)\otimes \id)(-y_i\circ y_j\otimes x_i \otimes x_j-y_j\otimes x_i\otimes y_i\odot x_j+y_j\odot x_j\vartriangleleft y_i\otimes x_i)\big)\\
&\quad-\sum\limits_{j}(\id\otimes L_{\vartriangleright}(a\odot x_{j})\otimes \id)(y_j\otimes(\tau r-r))\\
&=(\id\otimes L_{\vartriangleright}(a)\otimes \id)(-r_{21}\circ r_{31}-r_{23}\odot r_{31}+r_{21}\vartriangleleft r_{32})\\
&\quad-\sum\limits_{j}(\id\otimes L_{\vartriangleright}(a\odot x_{j})\otimes \id)(y_j\otimes(\tau r-r)).
\end{flalign*}
For $P_3(a)$, since
\begin{eqnarray*}
&&\sum\limits_{i,j}\big(-y_j\otimes x_i\otimes(a\star y_i)\odot x_j-y_i\otimes y_j\otimes(a\odot x_i)\odot x_j+y_i\otimes x_j\otimes (a\odot x_i)\star y_j\big)\\
&&\;\;=\sum\limits_{i,j}y_i\otimes x_j\otimes a\odot (y_j\vartriangleright x_i-x_i\vartriangleleft y_j)-\sum\limits_{i}(\id\otimes \id\otimes L_{\odot}(a\odot x_{i}))(y_i\otimes(\tau r-r)),
\end{eqnarray*}
we have
\begin{flalign*}
P_3(a)&=\sum\limits_{i,j}\big((\id\otimes \id\otimes L_{\odot}(a))(y_i\circ y_j\otimes x_j \otimes  x_i+y_j\otimes y_i\odot x_j\otimes  x_i\\
&\quad+y_i\otimes x_j\otimes y_j\vartriangleright x_i-y_i\otimes x_j\otimes x_i\vartriangleleft y_j+y_j\otimes y_i\vartriangleright x_j\otimes  x_i\\
&\quad+y_i\star y_j\otimes x_j\otimes  x_i)\big)-\sum\limits_{i}(\id\otimes \id\otimes L_{\odot}(a\odot x_{i}))((y_i\otimes(\tau r-r)),\\
&=(\id\otimes \id\otimes L_{\odot}(a))(r_{31}\circ r_{21}+r_{32}\odot r_{21}+r_{23}\vartriangleright r_{31}-r_{31}\vartriangleleft r_{23}+r_{32}\vartriangleright r_{21}+r_{31}\star r_{21})\\
&\quad-\sum\limits_{i}(\id\otimes \id\otimes L_{\odot}(a\odot x_{i}))(y_i\otimes(\tau r-r)).
\end{flalign*}
Therefore, Eq. (\ref{cob1}) is equivalent to Eq. \eqref{r1}.

Similarly, one can prove that Eq. (\ref{cob2}) is equivalent to Eq. \eqref{r2}, Eq. (\ref{cob3}) is equivalent to Eq. \eqref{r3}, and Eq. (\ref{cob4}) is equivalent to Eq. \eqref{r4}. Then the proof is completed.
\end{proof}

Next, we introduce the definition of pre-Novikov Yang-Baxter equation.
\begin{definition}
Let $(A,\vartriangleleft,\vartriangleright)$ be a pre-Novikov algebra and $r\in A\otimes A$. The following equation
\begin{flalign}
r_{12}\circ r_{13}+r_{23}\odot r_{13}-r_{12}\vartriangleleft r_{23}=0\label{yb}
\end{flalign}
is called the \textbf{pre-Novikov Yang-Baxter equation} in $(A,\vartriangleleft,\vartriangleright)$.
\end{definition}

\delete{\begin{lemma}\label{l4}
Let $(A,\vartriangleleft,\vartriangleright)$ be a pre-Novikov algebra and $r\in A\otimes A$ be a symmetric solution of pre-Novikov Yang-Baxter equation. Then we have $R_{ij}=0$, where $R_{ij}=0$ are defined by Lemma \ref{l3}, $i=1,2,3,4;j=1,2,3.$
\end{lemma}

\begin{proof}
Clearly, we know that $R_{41}=0$ since $r$ is symmetric. Suppose that $\sigma$ is an element in the permutation group $S_3$ acting on \{1,2,3\}. Then there is a linear map induced by $\sigma$ from $A\otimes A\otimes A $ to $A\otimes A\otimes A$. We still denote this linear map by $\sigma$.

Therefore, it is straightforward to deduce that $R_{11},R_{12},R_{32},R_{42}$ are equivalent to zero. For example, set $\sigma=(23)\in S_3$, then we have
\begin{flalign*}
\sigma(R_{11})&=\sigma(r_{21}\circ r_{31}-r_{21}\circ r_{32}-r_{31}\odot r_{32}+r_{21} \vartriangleright  r_{23}+r_{31}\star r_{23})\\
&=\sigma(r_{21}\circ r_{31}+r_{23}\odot r_{31}-r_{21}\vartriangleleft r_{32})\\
&=r_{\sigma(21)}\circ r_{\sigma(31)}+r_{\sigma(23)}\odot r_{\sigma(31)}-r_{\sigma(21)}\vartriangleleft r_{\sigma(32)}\\
&=r_{31}\circ r_{21}+r_{32}\odot r_{21}-r_{31}\vartriangleleft r_{23}\\
&=-R_{42}.
\end{flalign*}

For $R_{21}$, set $\sigma=(12)\in S_3$. Since $r$ is symmetric, we have
\begin{flalign*}
R_{21}&=r_{12}\vartriangleright r_{13}+r_{12}\vartriangleright r_{23}+r_{13}\star r_{23}\\
&=(r_{21}\circ r_{31}+r_{23}\odot r_{31}-r_{21}\vartriangleleft r_{32})+(r_{12}\circ r_{23}+r_{13}\odot r_{32}-r_{12}\vartriangleleft r_{13})\\
&=R_{41}+\sigma(R_{41})=0.
\end{flalign*}

Therefore, it is straightforward to deduce that $R_{13},R_{21},R_{22},R_{23},R_{31}$ are equivalent to zero.
\end{proof}}

\begin{theorem}\label{t3}
Let $(A,\vartriangleleft,\vartriangleright)$ be a pre-Novikov algebra and $r\in A\otimes A$ be a symmetric solution of pre-Novikov Yang-Baxter equation in $(A,\vartriangleleft,\vartriangleright)$. Let $\alpha,\beta:A\rightarrow A\otimes A$ be linear maps defined by Eqs. \eqref{e11} and \eqref{e12} respectively. Then $(A,\vartriangleleft,\vartriangleright,\alpha,\beta)$ is a pre-Novikov bialgebra.
\end{theorem}
\begin{proof}
Clearly, we have $R_{11}=-R_{12}=0$ since $r$ is symmetric. Suppose that $\sigma$ is an element in the permutation group $S_3$ acting on \{1,2,3\}. Then there is a linear map induced by $\sigma$ from $A\otimes A\otimes A $ to $A\otimes A\otimes A$. We still denote this linear map by $\sigma$.

Set $\sigma=(23)\in S_3$. Then we have
\begin{flalign*}
0&=\sigma(R_{12})\\
&=\sigma(-r_{21}\circ r_{31}-r_{23}\odot r_{31}+r_{21}\vartriangleleft r_{32})\\
&=-r_{\sigma(21)}\circ r_{\sigma(31)}-r_{\sigma(23)}\odot r_{\sigma(31)}+r_{\sigma(21)}\vartriangleleft r_{\sigma(32)}\\
&=-r_{31}\circ r_{21}-r_{32}\odot r_{21}+r_{31}\vartriangleleft r_{23}\\
&=R_{41}.
\end{flalign*}
For $R_{21}$, set $\sigma=(12)\in S_3$. Since $r$ is symmetric, we have
\begin{flalign*}
R_{21}&=r_{12}\vartriangleright r_{13}+r_{12}\vartriangleright r_{23}+r_{13}\star r_{23}\\
&=(r_{12}\circ r_{13}+r_{23}\odot r_{13}-r_{12}\vartriangleleft r_{23})+(r_{12}\circ r_{23}+r_{13}\odot r_{23}-r_{12}\vartriangleleft r_{13})\\
&=(r_{21}\circ r_{31}+r_{23}\odot r_{31}-r_{21}\vartriangleleft r_{32})+(r_{12}\circ r_{32}+r_{13}\odot r_{32}-r_{12}\vartriangleleft r_{31})\\
&=-R_{12}-\sigma(R_{12})=0.
\end{flalign*}

Similarly, one can verify that $R_{13}$, $R_{22}$ and $R_{31}$ are equal to zero.

Then this theorem follows directly by Lemmas \ref{l1} and \ref{l3}.
\end{proof}

Next, we investigate the operator forms of pre-Novikov Yang-Baxter equation. First, we introduce the definition of representations of pre-Novikov algebras.
\begin{definition}
Let $(A,\vartriangleleft,\vartriangleright)$ be a pre-Novikov algebra and $V$ be a vector space. Let  $l_{\vartriangleright},r_{\vartriangleright},l_{\vartriangleleft},r_{\vartriangleleft}:A\rightarrow \text{End}_{\bf k}(V)$ be linear maps. $(V,l_{\vartriangleright},r_{\vartriangleright},l_{\vartriangleleft},r_{\vartriangleleft})$ is called a \textbf{representation} of $(A,\vartriangleleft,\vartriangleright)$ if satisfies
\begin{flalign}
&l_{\vartriangleright}(a)l_{\vartriangleright}(b)v-l_{\vartriangleright}(b)l_{\vartriangleright}(a)v=l_{\vartriangleright}(a\circ b-b \circ a)v,  \label{rp1}\\
&l_{\vartriangleright}(a)l_{\vartriangleleft}(b)v-l_{\vartriangleleft}(b)l_{\vartriangleright}(a)v=l_{\vartriangleleft}(a\vartriangleright b-b\vartriangleleft a)v+l_{\vartriangleleft}(b)l_{\vartriangleleft}(a)v, \label{rp2}\\
&r_{\vartriangleright}(a\vartriangleright b)v=r_{\vartriangleright}(b)(r_{\vartriangleright}+r_{\vartriangleleft})(a)v+
l_{\vartriangleright}(a)r_{\vartriangleright}(b)v-r_{\vartriangleright}(b)(l_{\vartriangleleft}+l_{\vartriangleright})(a)v,\label{rp3}  \\
&r_{\vartriangleright}(a\vartriangleleft b)v=r_{\vartriangleleft}(b)r_{\vartriangleright}(a)v+l_{\vartriangleleft}(a)(r_{\vartriangleright}
+r_{\vartriangleleft})(b)v-r_{\vartriangleleft}(b)l_{\vartriangleleft}(a)v,\label{rp4}\\
&l_{\vartriangleright}(a)r_{\vartriangleleft}(b)v-r_{\vartriangleleft}(b)l_{\vartriangleright}(a)v=r_{\vartriangleleft}(a\circ b)v-r_{\vartriangleleft}(b)r_{\vartriangleleft}(a)v,\label{rp5}\\
&r_{\vartriangleright}(a)(r_{\vartriangleright}+r_{\vartriangleleft})(b)v=r_{\vartriangleleft}(b)r_{\vartriangleright}(a)v,\label{rp6}\\
&l_{\vartriangleleft}(a\vartriangleright b)v=r_{\vartriangleright}(b)(l_{\vartriangleright}+l_{\vartriangleleft})(a)v,\label{rp7}\\
&l_{\vartriangleright}(a\circ b)v=r_{\vartriangleleft}(b)l_{\vartriangleright}(a)v,\label{rp8}\\
&r_{\vartriangleleft}(a)r_{\vartriangleleft}(b)v=r_{\vartriangleleft}(b)r_{\vartriangleleft}(a)v,\label{rp9}\\
&l_{\vartriangleleft}(a\vartriangleleft b)v=r_{\vartriangleleft}(b)l_{\vartriangleleft}(a)v,\label{rp10}\qquad a,b\in A,v\in V,
\end{flalign}
where $a\circ b=a\vartriangleright b+a \vartriangleleft b$.
\end{definition}

Note that $(A,L_{\vartriangleright},R_{\vartriangleright},L_{\vartriangleleft},R_{\vartriangleleft})$ is a representation of $(A,\vartriangleleft,\vartriangleright)$.

\begin{proposition}
Let $(A,\vartriangleleft,\vartriangleright)$  be a pre-Novikov algebra. Let $V$ be a vector space and  $l_{\vartriangleright},r_{\vartriangleright},l_{\vartriangleleft},r_{\vartriangleleft}:A\rightarrow \text{End}_{\bf k}(V)$ be linear maps. Define two binary operation $\vartriangleleft$ and $ \vartriangleright$ on the direct sum $A\oplus V$ of the underlying vector spaces of $A$ and $V$ by
 \begin{eqnarray*}
&&(a+u)\vartriangleleft(b+v):=a\vartriangleleft b+l_{\vartriangleleft}(a)v+r_{\vartriangleleft}(b)u,\\ &&(a+u)\vartriangleright(b+v):=a \vartriangleright b+l_{\vartriangleright}(a)v+r_{\vartriangleright}(b)u,\;\;a, b\in A, u, v\in V.
\end{eqnarray*}
 Then $(V,l_{\vartriangleright},r_{\vartriangleright},l_{\vartriangleleft},r_{\vartriangleleft})$ is a representation of  $(A,\vartriangleleft,\vartriangleright)$ if and only if $(A\oplus V, \vartriangleleft,\vartriangleright)$ is a pre-Novikov algebra, which is called the {\bf semi-direct product} of $A$ and $V$ and denoted by $A\ltimes _{l_{\vartriangleright},r_{\vartriangleright},l_{\vartriangleleft},r_{\vartriangleleft}}V$.
\end{proposition}
\begin{proof}
It is straightforward.
\end{proof}
\begin{proposition}
Let $(A,\vartriangleleft,\vartriangleright)$ be a pre-Novikov algebra. If $(V,l_{\vartriangleright},r_{\vartriangleright},l_{\vartriangleleft},r_{\vartriangleleft})$ is a representation of $(A,\vartriangleleft,\vartriangleright)$, then $(V^*,l_{\vartriangleright}^*+l_{\vartriangleleft}^*+r_{\vartriangleright}^*+r_{\vartriangleleft}^*,r_{\vartriangleright}^*,
-(r_{\vartriangleright}^*+l_{\vartriangleleft}^*),-(r_{\vartriangleright}^*+r_{\vartriangleleft}^*))$ is also a representation of $(A,\vartriangleleft,\vartriangleright)$.
\end{proposition}
\begin{proof}
For all $a$, $b\in A$, $f\in V^*$ and $v\in V$, we have
\begin{eqnarray*}
&&\langle( l_{\vartriangleright}^*+l_{\vartriangleleft}^*+r_{\vartriangleright}^*+r_{\vartriangleleft}^* )(a)( l_{\vartriangleright}^*+l_{\vartriangleleft}^*+r_{\vartriangleright}^*+r_{\vartriangleleft}^*)(b)f-  ( l_{\vartriangleright}^*+l_{\vartriangleleft}^*+r_{\vartriangleright}^*+r_{\vartriangleleft}^*)(b) ( l_{\vartriangleright}^*+l_{\vartriangleleft}^*+r_{\vartriangleright}^*+r_{\vartriangleleft}^* )(a)f\\
&&\quad \;\;-( l_{\vartriangleright}^*+l_{\vartriangleleft}^*+r_{\vartriangleright}^*+r_{\vartriangleleft}^*  )(a\circ b-b \circ a)f,v\rangle\\
&&\;\;=\langle f,(l_{\vartriangleright}+l_{\vartriangleleft}+r_{\vartriangleright}+r_{\vartriangleleft})(b)((l_{\vartriangleright}
+l_{\vartriangleleft}+r_{\vartriangleright}+r_{\vartriangleleft})(a)v)\\
&&\quad\;\;-(l_{\vartriangleright}+l_{\vartriangleleft}+r_{\vartriangleright}+r_{\vartriangleleft})(a)((l_{\vartriangleright}
+l_{\vartriangleleft}+r_{\vartriangleright}+r_{\vartriangleleft})(b)v)
+(l_{\vartriangleright}+l_{\vartriangleleft}+r_{\vartriangleright}+r_{\vartriangleleft})(a\circ b-b \circ a)v \rangle.
\end{eqnarray*}

By Eqs. \eqref{rp1} and \eqref{rp8}, we obtain
\begin{eqnarray*}
&&l_{\vartriangleright}(a)l_{\vartriangleright}(b)v-l_{\vartriangleright}(b)l_{\vartriangleright}(a)v=l_{\vartriangleright}(a\circ b-b \circ a)v=r_{\vartriangleleft}(b)l_{\vartriangleright}(a)v-r_{\vartriangleleft}(a)l_{\vartriangleright}(b)v.
\end{eqnarray*}
Then we get
\begin{eqnarray}
\label{dual-resp-1}(l_{\vartriangleright}+r_{\vartriangleleft})(a)l_{\vartriangleright}(b)v=(l_{\vartriangleright}+r_{\vartriangleleft})(b)l_{\vartriangleright}(a)v.
\end{eqnarray}

By Eqs. \eqref{rp2}, \eqref{rp7} and \eqref{rp10}, one obtains
\begin{eqnarray*}
&&l_{\vartriangleright}(a)l_{\vartriangleleft}(b)v-l_{\vartriangleleft}(b)l_{\vartriangleright}(a)v\\
&&\;\;=l_{\vartriangleleft}(a\vartriangleright b-b\vartriangleleft a)v+l_{\vartriangleleft}(b)l_{\vartriangleleft}(a)v\\
&&\;\;=r_{\vartriangleright}(b)(l_{\vartriangleright}+l_{\vartriangleleft})(a)v-r_{\vartriangleleft}(a)l_{\vartriangleleft}(b)v
+l_{\vartriangleleft}(b)l_{\vartriangleleft}(a)v.
\end{eqnarray*}
Then we have  \begin{eqnarray}
\label{dual-resp-2}(r_{\vartriangleright}+l_{\vartriangleleft})(b)(l_{\vartriangleright}+l_{\vartriangleleft})(a)v
=(l_{\vartriangleright}+r_{\vartriangleleft})(a)l_{\vartriangleleft}(b)v.
 \end{eqnarray}
By Eqs. \eqref{dual-resp-1} and \eqref{dual-resp-2}, one gets
\begin{flalign*}
(l_{\vartriangleright}+r_{\vartriangleleft}+r_{\vartriangleright}+l_{\vartriangleleft})(b)(l_{\vartriangleright}+l_{\vartriangleleft})(a)v
=(l_{\vartriangleright}+r_{\vartriangleleft}+r_{\vartriangleright}+l_{\vartriangleleft})(a)(l_{\vartriangleright}+l_{\vartriangleleft})(b)v.
\end{flalign*}

By Eqs. \eqref{rp3}, \eqref{rp6} and \eqref{rp7}, we have
\begin{flalign*}
r_{\vartriangleright}(a\vartriangleright b)v&=r_{\vartriangleright}(b)(r_{\vartriangleright}+r_{\vartriangleleft})(a)v+
l_{\vartriangleright}(a)r_{\vartriangleright}(b)v-r_{\vartriangleright}(b)(l_{\vartriangleright}+l_{\vartriangleleft})(a)v\\
&=(l_{\vartriangleright}+r_{\vartriangleleft})(a)r_{\vartriangleright}(b)v-l_{\vartriangleleft}(a\vartriangleright b)v.
\end{flalign*}
Then we obtain \begin{eqnarray}
\label{dual-resp-3}(r_{\vartriangleright}+l_{\vartriangleleft})(a\vartriangleright b)v= (l_{\vartriangleright}+r_{\vartriangleleft})(a)r_{\vartriangleright}(b)v.
\end{eqnarray}

By Eqs. \eqref{rp4}, \eqref{rp6} and \eqref{rp10}, one obtains
\begin{flalign*}
r_{\vartriangleright}(a\vartriangleleft b)v&=r_{\vartriangleleft}(b)r_{\vartriangleright}(a)v+l_{\vartriangleleft}(a)(r_{\vartriangleright}+r_{\vartriangleleft})(b)v
-r_{\vartriangleleft}(b)l_{\vartriangleleft}(a)v\\
&=r_{\vartriangleright}(a)(r_{\vartriangleright}+r_{\vartriangleleft})(b)v+l_{\vartriangleleft}(a)(r_{\vartriangleright}
+r_{\vartriangleleft})(b)v-l_{\vartriangleleft}(a\vartriangleleft b)v.
\end{flalign*}
Then we get
 \begin{eqnarray}
\label{dual-resp-4} (r_{\vartriangleright}+l_{\vartriangleleft})(a\vartriangleleft b)v = (r_{\vartriangleright}+l_{\vartriangleleft})(a)(r_{\vartriangleright}+r_{\vartriangleleft})(b)v.
 \end{eqnarray}

By Eqs. \eqref{rp5}, \eqref{rp8} and \eqref{rp9}, we have
\begin{eqnarray}
\label{dual-resp-5} (l_{\vartriangleright}+r_{\vartriangleleft})(a \circ b)v=(l_{\vartriangleright}+r_{\vartriangleleft})(a)r_{\vartriangleleft}(b)v,
\end{eqnarray}
Then by Eqs. \eqref{dual-resp-3}, \eqref{dual-resp-4} and \eqref{dual-resp-5}, one gets
\begin{flalign*}
(r_{\vartriangleright}+l_{\vartriangleleft}+l_{\vartriangleright}+r_{\vartriangleleft})(a \circ b)v=(r_{\vartriangleright}+l_{\vartriangleleft}+l_{\vartriangleright}+r_{\vartriangleleft})(a)(r_{\vartriangleright}
+r_{\vartriangleleft})(b)v.
\end{flalign*}

Therefore one obtains
\begin{flalign*}
&(l_{\vartriangleright}+l_{\vartriangleleft}+r_{\vartriangleright}+r_{\vartriangleleft})(b)(l_{\vartriangleright}
+l_{\vartriangleleft}+r_{\vartriangleright}+r_{\vartriangleleft})(a)v\\
&\;\;\quad-(l_{\vartriangleright}+l_{\vartriangleleft}+r_{\vartriangleright}+r_{\vartriangleleft})(a)(l_{\vartriangleright}
+l_{\vartriangleleft}+r_{\vartriangleright}+r_{\vartriangleleft})(b)v\\
&\;\;\quad+(l_{\vartriangleright}+l_{\vartriangleleft}+r_{\vartriangleright}+r_{\vartriangleleft})(a\circ b-b \circ a)v\\
&\;\;=(l_{\vartriangleright}+r_{\vartriangleright}+l_{\vartriangleleft}+r_{\vartriangleleft})(b)(l_{\vartriangleright}+l_{\vartriangleleft})(a)v
-(l_{\vartriangleright}+r_{\vartriangleright}+l_{\vartriangleleft}+r_{\vartriangleleft})(a)(l_{\vartriangleright}+l_{\vartriangleleft})(b)v\\
&\;\;\quad+(l_{\vartriangleright}+r_{\vartriangleright}+l_{\vartriangleleft}+r_{\vartriangleleft})(b)(r_{\vartriangleright}+r_{\vartriangleleft})(a)v
-(l_{\vartriangleright}+r_{\vartriangleright}+l_{\vartriangleleft}+r_{\vartriangleleft})(b\circ a)v\\
&\;\;\quad+(l_{\vartriangleright}+r_{\vartriangleright}+l_{\vartriangleleft}+r_{\vartriangleleft})(a\circ b)v-(l_{\vartriangleright}+r_{\vartriangleright}+l_{\vartriangleleft}+r_{\vartriangleleft})(a)(r_{\vartriangleright}+r_{\vartriangleleft})(b)v\\
&\;\;=0.
\end{flalign*}
Hence Eq. \eqref{rp1} holds.

Similarly, one can check that Eqs. \eqref{rp2}-\eqref{rp10} hold. Then the proof is completed.
\end{proof}

We introduce the definition of $\mathcal{O}$-operators on a pre-Novikov algebra associated to a representation as follows.
\begin{definition}
Let $(A,\vartriangleleft,\vartriangleright)$ be a pre-Novikov algebra and $(V,l_{\vartriangleright},r_{\vartriangleright},l_{\vartriangleleft},r_{\vartriangleleft})$ be a representation of $(A,\vartriangleleft,\vartriangleright)$.  A linear map $T:V\rightarrow A$ is called an \textbf{$\mathcal{O}$-operator} on $(A,\vartriangleleft,\vartriangleright)$ associated to $(V,l_{\vartriangleright},r_{\vartriangleright},l_{\vartriangleleft},r_{\vartriangleleft})$ if $T$ satisfies
\begin{eqnarray}
&&T(u)\vartriangleright T(v)=T(l_{\vartriangleright}(T(u))v)+T(r_{\vartriangleright}(T(v))u),\\ &&T(u)\vartriangleleft T(v)=T(l_{\vartriangleleft}(T(u))v)+T(r_{\vartriangleleft}(T(v))u),\;\; u, v\in V.
\end{eqnarray}
\end{definition}

For a finite-dimensional vector space $A$, the natural isomorphism
\begin{eqnarray*}
A\otimes A\cong \text{Hom}_{\bf k}(A^\ast, A)
\end{eqnarray*}
identifies an $r\in A\otimes A$ with a linear map $T_r:A^*\rightarrow A$, which is defined by
\begin{flalign}
\langle f\otimes g,r\rangle=\langle f,T_r(g)\rangle, \quad  f, g\in A^*.\label{f1}
\end{flalign}

\begin{proposition} \label{p3}
Let $(A,\vartriangleleft,\vartriangleright)$ be a pre-Novikov algebra, $r\in A\otimes A$ be symmetric and $(A,\circ )$ be the associated Novikov algebra of  $(A,\vartriangleleft,\vartriangleright)$. Then $r$ is a solution of pre-Novikov Yang-Baxter equation in $(A,\vartriangleleft,\vartriangleright)$ if and only if $T_r$ is an $\mathcal{O}$-operator on $(A,\circ )$ associated to $(A^*,L_{\vartriangleright}^*+R_{\vartriangleleft}^*,-R_{\vartriangleleft}^*)$.
\end{proposition}
\begin{proof}
Let $\{e_1,e_2,\cdot\cdot\cdot,e_n\}$ be a basis of $A$ and $\{e_1^*,e_2^*,\cdot\cdot\cdot,e_n^*\}$ be the dual basis of $A^\ast$. Assume that $e_i\vartriangleright e_j=\sum\limits_{k=1}^nb_{ij}^ke_k$, $e_i\vartriangleleft e_j=\sum\limits_{k=1}^nc_{ij}^ke_k$ and $r=\sum\limits_{i=1}^n\sum\limits_{j=1}^na_{ij}e_i\otimes e_j$, where $b_{ij}^k$, $c_{ij}^k$, $a_{ij}\in {\bf k}$ and $a_{ij}=a_{ji}$. By Eq. \eqref{f1}, we have $T_r(e_i^*)=\sum\limits_{k=1}^na_{ik}e_k$. Hence we obtain
$$L_{\vartriangleright}^*(T_r(e_i^*))e_j^*=-\sum\limits_{t=1}^n\sum\limits_{l=1}^na_{it}b_{tl}^je_l^*,\quad L_{\vartriangleleft}^*(T_r(e_i^*))e_j^*=-\sum\limits_{t=1}^n\sum\limits_{l=1}^na_{it}c_{tl}^je_l^*,$$
$$R_{\vartriangleright}^*(T_r(e_i^*))e_j^*=-\sum\limits_{t=1}^n\sum\limits_{l=1}^na_{it}b_{lt}^je_l^*,\quad R_{\vartriangleleft}^*(T_r(e_i^*))e_j^*=-\sum\limits_{t=1}^n\sum\limits_{l=1}^na_{it}c_{lt}^je_l^*.$$
Then the coefficient of $e_k$ in
\begin{align*}
T_r(e_i^*)\circ T_r(e_j^*)-T_r((L_{\vartriangleright}^*+R_{\vartriangleleft}^*)(T_r(e_i^*))e_j^*-R_{\vartriangleleft}^*(T_r(e_j^*))e_i^*)
\end{align*}
is
\begin{align*}
\sum\limits_{t=1}^n\sum\limits_{l=1}^n\big(a_{it}a_{jl}(b_{tl}^k+c_{tl}^k)+a_{it}a_{lk}(b_{tl}^j+c_{lt}^j)-a_{jt}a_{lk}c_{lt}^i\big).
\end{align*}
Hence $T_r$ is an $\mathcal{O}$-operator on $(A,\circ )$ associated to $(A^*,L_{\vartriangleright}^*+R_{\vartriangleleft}^*,-R_{\vartriangleleft}^*)$ if and only if
\begin{eqnarray}
\label{o-op1}\sum\limits_{t=1}^n\sum\limits_{l=1}^n\big(a_{it}a_{jl}(b_{tl}^k+c_{tl}^k)+a_{it}a_{lk}(b_{tl}^j+c_{lt}^j)-a_{jt}a_{lk}c_{lt}^i\big)=0
\end{eqnarray}
holds for all $i$, $j\in \{1, \cdots, n\}$. The left-hand side of Eq. \eqref{o-op1} is just the coefficient of $e_k\otimes e_i\otimes e_j$ in
$r_{12}\circ r_{13}+r_{23}\odot r_{13}-r_{12}\vartriangleleft r_{23}$. Therefore this conclusion holds.
\end{proof}

\begin{proposition}\label{p4}
Let $(A,\vartriangleleft,\vartriangleright)$ be a pre-Novikov algebra, $r\in A\otimes A$ be symmetric and $(A,\circ )$ be the associated Novikov algebra of  $(A,\vartriangleleft,\vartriangleright)$. Then $T_r$ is an $\mathcal{O}$-operator on $(A,\vartriangleleft,\vartriangleright)$ associated to $(A^*,L_{\vartriangleright}^*+L_{\vartriangleleft}^*+R_{\vartriangleright}^*+R_{\vartriangleleft}^*,R_{\vartriangleright}^*,
-(R_{\vartriangleright}^*+L_{\vartriangleleft}^*),-(R_{\vartriangleright}^*+R_{\vartriangleleft}^*))$ if and only if the following equations hold:
\begin{align}
&r_{12}\vartriangleright r_{13}+r_{23}\star r_{13}+r_{12}\vartriangleright r_{23}=0, \label{e16} \\
&r_{12} \vartriangleleft r_{13}-r_{13} \odot r_{23}-r_{12}\circ r_{23}=0. \label{e17}
\end{align}
\end{proposition}
\begin{proof}
It follows by a similar proof as that in Proposition \ref{p3}.
\end{proof}

\begin{theorem}\label{co2}
Let $(A,\vartriangleleft,\vartriangleright)$ be a pre-Novikov algebra, $r\in A\otimes A$ be symmetric and $(A,\circ )$ be the associated Novikov algebra of  $(A,\vartriangleleft,\vartriangleright)$. Then the following conditions are equivalent:
\begin{enumerate}
\item $r$ is a solution of pre-Novikov Yang-Baxter equation in $(A,\vartriangleleft,\vartriangleright)$;
\item $T_r$ is an $\mathcal{O}$-operator on $(A,\circ )$ associated to $(A^*,L_{\vartriangleright}^*+R_{\vartriangleleft}^*,-R_{\vartriangleleft}^*)$;
\item $T_r$ is an $\mathcal{O}$-operator on $(A,\vartriangleleft,\vartriangleright)$ associated to $(A^*,L_{\vartriangleright}^*+L_{\vartriangleleft}^*+R_{\vartriangleright}^*+R_{\vartriangleleft}^*,R_{\vartriangleright}^*,
-(R_{\vartriangleright}^*+L_{\vartriangleleft}^*),-(R_{\vartriangleright}^*+R_{\vartriangleleft}^*))$.
 \end{enumerate}
\end{theorem}
\begin{proof}
By a similar proof as that in Theorem \ref{t3}, Eq. \eqref{e17} holds if and only if Eq. \eqref{yb} holds.
Since $r$ is symmetric, Eq. \eqref{e16} holds if and only if $R_{21}$ holds, where $R_{21}$ is given in Lemma \ref{l3}. Then by Theorem \ref{t3}, Eq. \eqref{e16} holds if \eqref{yb} holds. Therefore, by Propositions \ref{p3} and \ref{p4}, we obtain the conclusion.
\end{proof}

\begin{theorem}\label{t2}
Let $(A,\vartriangleleft,\vartriangleright)$ be a pre-Novikov algebra and $(A,\circ )$ be the associated Novikov algebra of  $(A,\vartriangleleft,\vartriangleright)$. Let $(V,l_{\vartriangleright},r_{\vartriangleright},l_{\vartriangleleft},r_{\vartriangleleft})$ be a representation of $(A,\vartriangleleft,\vartriangleright)$. Let $T:V\rightarrow A$ be a linear map, which is identified with $r_T\in A\otimes V^\ast\subseteq (A\ltimes_{l_{\vartriangleright}^*+l_{\vartriangleleft}^*+r_{\vartriangleright}^*+r_{\vartriangleleft}^*,r_{\vartriangleright}^*,
-(r_{\vartriangleright}^*+l_{\vartriangleleft}^*),-(r_{\vartriangleright}^*+r_{\vartriangleleft}^*)}V^*)\otimes (A\ltimes_{l_{\vartriangleright}^*+l_{\vartriangleleft}^*+r_{\vartriangleright}^*+r_{\vartriangleleft}^*,r_{\vartriangleright}^*,
-(r_{\vartriangleright}^*+l_{\vartriangleleft}^*),-(r_{\vartriangleright}^*+r_{\vartriangleleft}^*)}V^*) $ through $\text{Hom}_{\bf k}(V, A)\cong A\otimes V^\ast$. Then $r=r_T+\tau (r_T)$ is a symmetric solution of pre-Novikov Yang-Baxter equation in $A\ltimes_{l_{\vartriangleright}^*+l_{\vartriangleleft}^*+r_{\vartriangleright}^*+r_{\vartriangleleft}^*,r_{\vartriangleright}^*,
-(r_{\vartriangleright}^*+l_{\vartriangleleft}^*),-(r_{\vartriangleright}^*+r_{\vartriangleleft}^*)}V^*$ if and only if $T$ is an $\mathcal{O}$-operator on $(A,\vartriangleleft,\vartriangleright)$ associated to $(V,l_{\vartriangleright},r_{\vartriangleright},l_{\vartriangleleft},r_{\vartriangleleft})$.
\end{theorem}

\begin{proof}
Let $\{v_1,v_2,\cdot\cdot\cdot,v_n\}$ be a basis of $V$ and $\{ v^*_1, v^*_2,\cdot\cdot\cdot,v^*_n\}$ be the dual basis of $V^\ast$. Then we have
\begin{flalign*}
r_T&=\sum\limits_{i=1}^nT(v_i)\otimes v^*_i\in T(V)\otimes V^* \subset(A\oplus V^*)\otimes(A\oplus V^*),\\
r&=r_T+\tau (r_T)=\sum\limits_{i=1}^n(T(v_i)\otimes v^*_i+v^*_i\otimes T(v_i)).
\end{flalign*}
Therefore, we have
\begin{flalign*}
r_{12}\circ r_{13}&=\sum\limits_{i,j=1}^n\big(T(v_i)\circ T(v_j)\otimes v^*_i\otimes v^*_j+(l^*_{\vartriangleright}+r^*_{\vartriangleleft})(T(v_i))v^*_j\otimes v^*_i\otimes T(v_j)\\
&\quad-r_{\vartriangleleft}^*(T(v_j))v^*_i \otimes T( v_i) \otimes  v^*_j\big).
\end{flalign*}
By Eq. \eqref{e14}, we have
\begin{flalign*}
l^*_{\vartriangleright}(T(v_i))v^*_j=-\sum\limits_{k=1}^nv^*_j(l_{\vartriangleright}(T(v_i))v_k)v^*_k,\quad r^*_{\vartriangleright}(T(v_i))v^*_j=-\sum\limits_{k=1}^nv^*_j(r_{\vartriangleright}(T(v_i))v_k)v^*_k,\\
l^*_{\vartriangleleft}(T(v_i))v^*_j=-\sum\limits_{k=1}^nv^*_j(l_{\vartriangleleft}(T(v_i))v_k)v^*_k,\quad
r^*_{\vartriangleleft}(T(v_i))v^*_j=-\sum\limits_{k=1}^nv^*_j(r_{\vartriangleleft}(T(v_i))v_k)v^*_k.
\end{flalign*}
Hence we obtain
\begin{flalign*}
\sum\limits_{i=1}^n\sum\limits_{j=1}^n(l^*_{\vartriangleright}+r^*_{\vartriangleleft})(T(v_i))v^*_j\otimes v^*_i\otimes T(v_j)&=-\sum\limits_{i=1}^n\sum\limits_{j=1}^nv^*_j\otimes v^*_i\otimes T((l_{\vartriangleright}+r_{\vartriangleleft})(T(v_i))v_j),\\
\sum\limits_{i=1}^n\sum\limits_{j=1}^nr_{\vartriangleleft}^*(T(v_j))v^*_i \otimes T( v_i) \otimes  v^*_j&=-\sum\limits_{i=1}^n\sum\limits_{j=1}^nv^*_i \otimes T(r_{\vartriangleleft}(T(v_j)) v_i) \otimes  v^*_j.
\end{flalign*}
Then one gets
\begin{flalign*}
r_{12}\circ r_{13}&=\sum\limits_{i=1}^n\sum\limits_{j=1}^n\big(T(v_i)\circ T(v_j)\otimes v^*_i\otimes v^*_j-v^*_j\otimes v^*_i\otimes T((l_{\vartriangleright}+r_{\vartriangleleft})(T(v_i))v_j)\\
&\quad+v^*_i \otimes T(r_{\vartriangleleft}(T(v_j)) v_i) \otimes  v^*_j\big).
\end{flalign*}

Similarly, we obtain
\begin{flalign*}
r_{23} \vartriangleright r_{13}&=\sum\limits_{i=1}^n\sum\limits_{j=1}^n\big( v^*_j\otimes T(v_i) \otimes r_{\vartriangleright}^*(T(v_j))v^*_i+T(v_j)\otimes v^*_i\otimes (l_{\vartriangleright}^*+l_{\vartriangleleft}^*+r_{\vartriangleright}^*+r_{\vartriangleleft}^*)(T(v_i))v^*_j\\
&\quad+v^*_j \otimes v^*_i \otimes T(v_i)\vartriangleright T(v_j)\big)\\
&=\sum\limits_{i=1}^n\sum\limits_{j=1}^n\big( -v^*_j\otimes T( r_{\vartriangleright}(T(v_j))v_i) \otimes v^*_i-T((l_{\vartriangleright}+l_{\vartriangleleft}+r_{\vartriangleright}+r_{\vartriangleleft})(T(v_i))v_j)\otimes v^*_i\otimes v^*_j\\
&\quad+v^*_j \otimes v^*_i \otimes T(v_i)\vartriangleright T(v_j)\big),\\
r_{13}\vartriangleleft r_{23}&=\sum\limits_{i=1}^n\sum\limits_{j=1}^n\big(T(v_i)\otimes v^*_j \otimes (-r_{\vartriangleright}^*-r_{\vartriangleleft}^*)(T(v_j))v^*_i+ v^*_i \otimes T(v_j)\otimes(-r_{\vartriangleright}^*-l_{\vartriangleleft}^*)(T(v_i)) v^*_j\\
& \quad+v^*_i\otimes v^*_j\otimes T(v_i)\vartriangleleft T(v_j)\big)\\
&=\sum\limits_{i=1}^n\sum\limits_{j=1}^n\big(T((r_{\vartriangleright}+r_{\vartriangleleft})(T(v_j))v_i)\otimes v^*_j \otimes v^*_i+ v^*_i \otimes T((r_{\vartriangleright}+l_{\vartriangleleft})(T(v_i))v_j)\otimes v^*_j\\
& \quad+v^*_i\otimes v^*_j\otimes T(v_i)\vartriangleleft T(v_j)\big),\\
-r_{12}\vartriangleleft r_{23}&=-\sum\limits_{i=1}^n\sum\limits_{j=1}^n\big( T(v_i)\otimes (-r_{\vartriangleright}^*-r_{\vartriangleleft}^*)(T(v_j))v^*_i \otimes v^*_j+v^*_i \otimes  T(v_i)\vartriangleleft T(v_j)\otimes v^*_j \\
&\quad+v^*_i\otimes(-r_{\vartriangleright}^*-l_{\vartriangleleft}^*)(T(v_i)) v^*_j\otimes T(v_j) \big)\\
&=\sum\limits_{i=1}^n\sum\limits_{j=1}^n\big( T((-r_{\vartriangleright}-r_{\vartriangleleft})(T(v_j))v_i)\otimes v^*_i \otimes v^*_j-v^*_i \otimes  T(v_i)\vartriangleleft T(v_j)\otimes v^*_j \\
&\quad+v^*_i\otimes v^*_j\otimes T((-r_{\vartriangleright}-l_{\vartriangleleft})(T(v_i)) v_j) \big).
\end{flalign*}
Therefore, $r=r_T+\tau (r_T)$ is a symmetric solution of pre-Novikov Yang-Baxter equation in $A\ltimes_{l_{\vartriangleright}^*+l_{\vartriangleleft}^*+r_{\vartriangleright}^*+r_{\vartriangleleft}^*,r_{\vartriangleright}^*,
-(r_{\vartriangleright}^*+l_{\vartriangleleft}^*),-(r_{\vartriangleright}^*+r_{\vartriangleleft}^*)}V^*$ if and only if
\begin{eqnarray}
&&\label{o1}T(v_i)\circ T(v_j)=T((l_{\vartriangleright}+l_{\vartriangleleft})(T(v_i))v_j)+T((r_{\vartriangleright}+r_{\vartriangleleft})(T(v_j))v_i),\\
&&\label{o2}T(v_i)\vartriangleleft T(v_j)=T(l_{\vartriangleleft}(T(v_i))v_j)+T(r_{\vartriangleleft}(T(v_j))v_i),\\
&&\label{o3}T(v_i)\vartriangleright T(v_j)+T(v_j)\vartriangleleft T(v_i)=T((l_{\vartriangleright}+r_{\vartriangleleft})(T(v_i))v_j)+T((r_{\vartriangleright}+l_{\vartriangleleft})(T(v_j))v_i),
\end{eqnarray}
hold for all $i$, $j\in \{1, \cdots, n\}$. It is easy to see that Eqs. \eqref{o1}-\eqref{o3} hold for all $i$, $j\in \{1, \cdots, n\}$ if and only if $T$ is an $\mathcal{O}$-operator on $(A,\vartriangleleft,\vartriangleright)$ associated to $(V,l_{\vartriangleright},r_{\vartriangleright},l_{\vartriangleleft},r_{\vartriangleleft})$.
Then the proof is completed.\end{proof}

Finally, we present an example to construct pre-Novikov bialgebras by $\mathcal{O}$-operators.

\begin{example}\label{ex2}
Let $(A=\mathbf{k}e_1\oplus \mathbf{k}e_2,\vartriangleleft,\vartriangleright)$ be the pre-Novikov algebra given in Example \ref{ex1}.
Let $T: A\rightarrow A$ be a linear map given by
$$T(e_1)=e_2,\quad T(e_2)=0.$$
Then it is easy to check that $T$ is an $\mathcal{O}$-operator on $(A,\vartriangleleft,\vartriangleright)$ associated to $(A,L_{\vartriangleright},R_{\vartriangleright},L_{\vartriangleleft},R_{\vartriangleleft})$.

Denote the semi-direct product of $(A, \vartriangleleft,\vartriangleright)$ and its  representation $(A^*,L_{\vartriangleright}^*+L_{\vartriangleleft}^*+R_{\vartriangleright}^*+R_{\vartriangleleft}^*,R_{\vartriangleright}^*,$ $
-(R_{\vartriangleright}^*+L_{\vartriangleleft}^*),-(R_{\vartriangleright}^*+R_{\vartriangleleft}^*))$
by $(B,\vartriangleleft,\vartriangleright)$. In fact, $B$ is the vector space $\mathbf{k}e_1\oplus\mathbf{k}e_2\oplus\mathbf{k}e_1^*\oplus\mathbf{k}e_2^*$ endowed with two binary operations $\vartriangleright,\vartriangleleft$ defined by non-zero products:
\begin{flalign*}
&e_1\vartriangleleft e_1=e_1,\quad e_1\vartriangleleft e_2= e_2\vartriangleleft e_1=e_2,\\
&e_1\vartriangleright e_1^*=e_2\vartriangleright e_2^*=-2e_1^*,\qquad e_1\vartriangleright e_2^*=-2e_2^*,\\
&e_1\vartriangleleft e_1^*=e_2\vartriangleleft e_2^*=e_1^*\vartriangleleft e_1=e_2^*\vartriangleleft e_2=e_1^*,\quad e_1\vartriangleleft e_2^*=e_2^*\vartriangleleft e_1=e_2^*.
\end{flalign*}

By Theorem \ref{t2}, $r= \sum\limits_{i=1}^2 \sum\limits_{j=1}^2(T(e_i)\otimes e_i^*+ e_i^*\otimes T(e_i))=e_2\otimes e_1^*+e_1^*\otimes e_2$ is a symmetric solution of pre-Novikov Yang-Baxter equation in the pre-Novikov algebra $(B,\vartriangleleft,\vartriangleright)$. Let $(B, \circ)$ be the associated Novikov algebra of $(B,\vartriangleleft,\vartriangleright)$.
Then by Theorem \ref{t3}, $(B,\vartriangleleft,\vartriangleright,\alpha,\beta)$ is a pre-Novikov bialgebra where $\alpha,\beta :B \rightarrow B\otimes B$ are linear maps given by
\begin{flalign*}
\alpha(a):&=(L_{\circ}(a)\otimes \id+\id \otimes(L_{\vartriangleright}+R_{\vartriangleleft})(a))\tau r\\
&=a\circ e_2\otimes e_1^*+ e_2\otimes(a\vartriangleright e_1^*+e_1^*\vartriangleleft a) +a\circ e_1^* \otimes e_2+ e_1^*\otimes(a\vartriangleright e_2+e_2\vartriangleleft a),\\
\beta(a):&=-(L_{\vartriangleright}(a)\otimes \id+\id \otimes (L_{\circ}+R_{\circ})(a))r\\
&=-a\vartriangleright e_2\otimes e_1^*- e_2\otimes(a\circ e_1^*+e_1^*\circ a) -a\vartriangleright e_1^* \otimes e_2- e_1^*\otimes(a\circ e_2+e_2\circ a),\;\;a\in B.
\end{flalign*}
Explicitly,
$\alpha(e_2^*)= 2e_1^*\otimes e_1^*,\beta(e_1)=-e_2\otimes e_1^*$, and the others are zero.
\end{example}

\noindent{\bf Acknowledgments.} This research is supported by
NSFC (12171129).

\smallskip

\noindent
{\bf Declaration of interests. } The authors have no conflicts of interest to disclose.

\smallskip

\noindent
{\bf Data availability. } Data sharing is not applicable to this article as no new data were created or analyzed in this study.

\end{document}